\documentclass[reqno]{amsart}
\usepackage{xcolor}
\usepackage{amsthm,amsmath,amssymb,latexsym,soul,cite,mathrsfs}
\usepackage{enumitem}
\usepackage{color,enumitem,graphicx}
\usepackage[colorlinks=true,urlcolor=blue,
citecolor=red,linkcolor=blue,linktocpage,pdfpagelabels,
bookmarksnumbered,bookmarksopen]{hyperref}
\usepackage[english]{babel}

\usepackage[left=2.7cm,right=2.7cm,top=2.9cm,bottom=2.9cm]{geometry}

\usepackage[hyperpageref]{backref}

\numberwithin{equation}{section}
\newtheorem{theorem}{Theorem}[section]
\theoremstyle{plain}
\newtheorem{theoremletter}{Theorem}

\newtheorem{lemma}[theorem]{Lemma}
\newtheorem{lemmaletter}[theoremletter]{Lemma}
\newtheorem{corollary}[theorem]{Corollary}
\newtheorem{proposition}[theorem]{Proposition}

\newtheorem{remark}[theorem]{Remark}

\newcommand{\dx}{\,\mathrm{d}x}
\newcommand{\dt}{\,\mathrm{d}t}
\newcommand{\dy}{\,\mathrm{d}y}
\newcommand{\dxi}{\,\mathrm{d}\xi}
\newcommand{\dtau}{\,\mathrm{d}\tau}

\newcommand{\dxdy}{\,\mathrm{d}x\mathrm{d}y}

\newcommand{\dive}{\mathrm{div}}

\newcommand{\loca}{\operatorname{loc}}

\usepackage[norefs,nocites]{refcheck}

\title[Rayleigh quotient and general nonlinearities]{Ground states of nonlocal elliptic equations with general nonlinearities via Rayleigh quotient}

\author[D.~Ferraz]{Diego Ferraz}
\address{Department of Mathematics,
	Federal University of Rio Grande do Norte,
	59078-970, Natal-RN, Brazil}
\email{diego.ferraz.br@gmail.com}

\author[Edcarlos D. Silva]{Edcarlos D. Silva}
\address{Department of Mathematics,
	Federal University of Goi\'{a}s,
	74001-970, Goi\'{a}s-GO, Brazil}
\email{edcarlos@ufg.br}

\thanks{Corresponding author: Edcarlos D. Silva. email: edcarlos@ufg.br}
\thanks{The second author was also partially supported by CNPq with grant 309026/2020-2.}
\begin{document}
	\begin{abstract}
		It is established ground states and multiplicity of solutions for a nonlocal Schr\"{o}dinger equation
		$(-\Delta )^s u + V(x) u = \lambda a(x) |u|^{q-2}u + b(x)f(u)$ in $\mathbb{R}^N,$ $u \in H^s(\mathbb{R}^N),$ where $0<s<\min\{1,N/2\},$ $1<q<2$ and $\lambda >0,$ under general conditions over the measurable functions $a,$ $b$, $V$ and $f.$ The nonlinearity $f$ is superlinear at infinity and at the origin, and does not satisfy any Ambrosetti-Rabinowitz type condition. It is considered that the weights $a$ and $b$ are not necessarily bounded and the potential $V$ can change sign. We obtained a sharp $\lambda^*> 0$ which guarantees the existence of at least two nontrivial solutions for each $\lambda \in (0, \lambda^*)$. Our approach is variational in its nature and is based on the nonlinear Rayleigh quotient method together with some fine estimates. Compactness of the problem is also considered.
	\end{abstract}
	\maketitle

	\section{Introduction}
	Our main concern in this paper is to study ground states for the following nonlocal elliptic problem
	\begin{equation}\label{P}\tag{$P$}
		\left\{
		\begin{aligned}
			&(-\Delta )^s u + V(x) u = \lambda a(x) |u|^{q-2}u + b(x)f(u)\quad \text{in}\quad \mathbb{R}^N,\\
			&u \in H^s(\mathbb{R}^N),
		\end{aligned}
		\right.
	\end{equation}
	involving measurable functions $a,$ $b,$ $V$ and $f$ under general hypotheses to be stated later on, where $0<s<\min\{1,N/2\},$ $\lambda >0$ and $1<q<2.$ Here the fractional Laplacian $(-\Delta )^s u$ is defined by the relation	
	\begin{equation*}
		\mathscr{F} \left((-\Delta )^{s} u \right) (\xi) = \left|\xi\right| ^{2s} \mathscr{F}u (\xi),\quad \xi \in \mathbb{R}^N,
	\end{equation*}
	where  $\mathscr{F}u$ is the Fourier transform of $u,$ i.e.
	\begin{equation*}
		\mathscr{F} u (x)= \frac{1}{(2\pi)^{N/2}} \int _{\mathbb{R}^N} u (\xi) e ^{- i \xi \cdot x} \dxi,\quad x \in \mathbb{R}^N.
	\end{equation*}
	Equivalently, if $u\in \mathscr{S}$ (Schwartz space) the fractional Laplacian of $u$ can be computed by the following singular integral
	\begin{equation*}
		(-\Delta )^{s} u (x) = C(N,s) \lim _{\varepsilon \rightarrow 0 ^{+} } \int _{ \mathbb{R}^N \setminus B_{\varepsilon } (0) } \frac{u(x) - u(y)}{|x-y|^{N + 2s}}\dy,
	\end{equation*}
	for a suitable positive normalizing constant $C(N,s).$ 
	
	Nonlocal elliptic problems involving the fractional Laplacian have been widely considered in the last years. In fact, this class of problems arises naturally in several
	branches of mathematical physics. For instance, solutions of \eqref{P} can be seen as stationary states (corresponding to solitary waves) of nonlinear Schr\"{o}dinger equations of the form 
	\begin{equation*}
		i \phi _t - (-\Delta^s )\phi + V_0(x)\phi + g(x,u) =0, \quad \text{in }\mathbb{R}^N,
	\end{equation*}
	for suitable $g$ and $V_0.$ For more about the fractional Laplacian operator and its applications we cite \cite{landkof,hitchhiker}.

	Recently, semilinear elliptic problems have been extensively studied by a variational point of view by considering the so called nonlinear Rayleigh quotient method, see \cite{yavdat2017, yavdat2022,Lean,Kaye,Giovany,carvalho-silva-goulart2021}. In \cite{yavdat2017} Y. Ilyasov 
	develops a survey involving the theory of Rayleigh’s quotient for nonlinear problems. Roughly speaking, the author shows that this method can be applied when it is not possible to use the Nehari method directly. For instance, denoting $J_\lambda$ the related energy functional with respect to \eqref{P}, it is crucial in the Nehari manifold method that $t \mapsto J_\lambda (tu)$ has a unique critical point. This is not the case for $J_\lambda,$ where in fact it has two distinct critical points (see Proposition \ref{compar}), under our conditions.
	
	In \cite{Kaye}, K. Silva studies an abstract equation in a reflexive Banach space, depending on a real parameter $\lambda,$ where this equation is composed by suitable homogeneous operators. By analyzing the Nehari sets, and applying the nonlinear Rayleigh quotient method \cite{yavdat2017}, he proves a general bifurcation result. On the other hand, in \cite{Lean}, M. C. Carvalho et al. studied the following problem
	\begin{equation*}
		\left\{
		\begin{aligned}
			&-\Delta_p u = |u|^{\gamma -2}u + \mu |u|^{\alpha -2 } u -\lambda |u|^{q-2} u &\text{in }& \Omega,\\
			&u = 0 &\text{on }& \partial \Omega,
		\end{aligned}
		\right.
	\end{equation*}
	where $\Omega$ is a smooth bounded domain in $\mathbb{R}^N,$ $N \geq 1,$ $\Delta_p u = \dive (|\nabla u|^{p-2}\nabla u) $ is the $p$--Laplacian operator, under suitable assumptions on the parameters $q,$ $\alpha,$ $p,$ $\gamma,$ $\lambda $ and $\mu.$ The authors discuss multiplicity of positive solutions leading to the occurrence of an $S$-shaped bifurcation curve. They deal with relatively unexplored cases by using the nonlinear generalized Rayleigh quotient method, and they find a range of parameters where the equation may have distinct branches of solutions. For further results using the nonlinear Rayleigh quotient method we refer the reader to \cite{Faraci,Carvalho}. 
	
	It is worth to mention the important fact that in all above works the corresponding energy functional possesses a homogeneity in some degree on each of its terms, which appears to be a fundamental condition in order to apply variational arguments and the approach via Rayleigh quocient method, simultaneously. Based on this, a natural question arises: Is it possible to study existence of ground states for Problem \eqref{P} via Rayleigh quotient method \textit{without using homogeneous conditions} in all of its terms? In this paper we give a positive answer for this question, by introducing a new set of general hypotheses over the nonlinearity $f$ (see \ref{f_um}--\ref{f_quatro} in Sect. \ref{s_Hmainresults}), where $f$ is superlinear at the origin and at infinity, not necessarily homogeneous and not requiring any Ambrosetti-Rabinowitz type condition.
	
	On the other hand, one of the main difficulties to study problems like \eqref{P} by means of variational methods lies in the lack of compactness, which, roughly speaking, originates from the invariance of $\mathbb{R}^N$ with respect to translation and dilation and, analytically, appears because of noncompactness of the corresponding Sobolev embedding. Usually one can suppose that $V$ is bounded from below by a positive constant $V_0 >0$ and that $V$ is coercive in the Bartsch-Wang sense \cite{BW} to ensure that the natural related Sobolev space $X$ for the problem \eqref{P} is compact embedded into a suitable Lebesgue space $L$. See also \cite{chen-gao2018,zhang-tang-chen2021} and different conditions in \cite{Felmer,Secchi,Bisci,sirakov2000}. Differently of the above cited works, and inspired by \cite{sirakov2000,desouza2020}, we consider a changing sign potential $V$ that is only bounded from below by a negative constant. The potential $V$ is assumed in a such way that together with another hypotheses, ensures compactness of the embedding of $X$ in $L$ (see \ref{V_sirakov1}--\ref{V_sirakov3} in Sect. \ref{s_Hmainresults}). Nevertheless, an other novelty of our work is that we also consider the weights $a$ and $b,$ present in the right-hand side of Eq. \eqref{P}, not necessarily bounded. In fact, notice that the term $\lambda a(x) |u|^{q-2}u$ introduces some additional difficulties in order to control the energy functional of problem \eqref{P} and the associated Rayleigh quotient. The first one is to show how $\int_{\mathbb{R}^N} a(x) |u|^{q-2}u \dx$ is finite due to the fact that $q \in (1, 2)$. The second one is to ensure that the energy functional for problem \eqref{P} and the associated Rayleigh quotient are well defined and are in $C^1$ class. 
	
	In our work we prove the existence of a sharp $\lambda^*> 0,$ in the sense that if $\lambda\geq \lambda ^\ast,$ then the argument involving minimization over the Nehari manifold is in general, not suitable anymore, since one cannot apply the Lagrange multiplier theorem in a direct way (for more details see Sect. \ref{s_main} or \cite{yavdat2017}). We also prove existence of at least one ground state of \eqref{P} provided $0<\lambda<\lambda ^\ast$ (Theorem \ref{th1}). Moreover, the existence of another extremal value $0<\lambda _\ast<\lambda ^\ast$ is obtained in order to characterize a second solution (not necessarily ground state), for suitable value of $0<\lambda<\lambda ^\ast$ (Theorem \ref{th2}). To the best of our knowledge, the present work is the first one that considers an approach to a class of elliptic problems via the nonlinear Rayleigh quotient and the Nehari method where $f$ is not a powerlike function in the presence of a changing sign potential $V$ with unbounded weights $a$ and $b.$

	\subsection{Hypotheses}\label{s_Hmainresults} Initially, inspired in part by \cite{sirakov2000,BW}, we consider the following assumptions on $V$:
	
	\begin{enumerate}[label=($V_1$),ref=$(V_1)$]
		\item \label{V_sirakov1} There exists $B\geq 0$ such that
		\begin{equation*}
			V(x) \geq - B,\quad \text{almost everywhere (a.e.)} \ x \in \mathbb{R}^N.
		\end{equation*}
	\end{enumerate}
	\begin{enumerate}[label=($V_2$),ref=$(V_2)$]
		\item\label{V_sirakov2} $V\in L^\infty  _{\loca}(\mathbb{R}^N)$ and  
		\begin{equation*}
			\inf \left\{   \int _{\mathbb{R}^N} |\xi |^{2s} |\mathscr{F} u |^2 \dxi +\int _{\mathbb{R}^N}  V(x)|u|^2 \dx  : u \in C^\infty _0 (\mathbb{R}^N) \text{ and }\| u \|_2 =1 \right\}>0.
		\end{equation*}
	\end{enumerate}
	Let $\Omega \subset \mathbb{R}^N$ be an open set with smooth boundary and $2\leq \theta < 2^\ast _s=2N/(N-2s).$ Define
	\begin{equation*}
		\nu_{\theta }(\Omega ) := \inf \left\{  \int _{\mathbb{R}^N} |\xi |^{2s} |\mathscr{F} u |^2 \dxi  + \int _{\mathbb{R}^N} V(x) |u|^2 \dx : u \in \mathcal{M}_\theta (\Omega )\right\},
	\end{equation*}
	where 
	\begin{equation*}
		\mathcal{M}_\theta (\Omega ) = \left\{ u \in H^s(\mathbb{R}^N ): u =0 \text{ in }\mathbb{R}^N \setminus \Omega \text{ and }\| u \|_{\theta} = 1\right\},
	\end{equation*}
	with $\nu(\emptyset) := + \infty. $ 
	\begin{enumerate}[label=($V_3$),ref=$(V_3)$]
		\item \label{V_sirakov3} There is $2\leq \theta_0 < 2^\ast _s $ such that
		\begin{equation*}
			\lim_{R \rightarrow  \infty}\nu_{\theta _0}(\mathbb{R}^N \setminus \overline{B}_R) = + \infty.	
		\end{equation*}
	\end{enumerate}
	For the function $b$ we shall consider the following hypothesis:
	\begin{enumerate}[label=($B_1$),ref=$(B_1)$]
		\item \label{B_weight} $b \in L^\infty _{\loca}(\mathbb{R}^N),$ $b(x)\geq 1$ a.e. in $\mathbb{R}^N$ and there are $C_0,$ $R_0>0$ such that
		\begin{equation*}
			b(x)  \leq C_0 \left( 1+ (V^+ (x)) ^{1/\alpha} \right),\quad \forall\ |x|\geq R_0\text{ and }\alpha >1,
		\end{equation*}
		where $V^+(x) = \max\{ 0, V(x)  \}.$
	\end{enumerate}
	We assume that $f \in C^1 (\mathbb{R}, \mathbb{R})$ and the following conditions:
	\begin{enumerate}[label=($f_1$),ref=$(f_1)$]
		\item\label{f_um} There are $C>0$ and $p \in (2,2^\ast _s)$ such that 
		\begin{equation*}
			|f(t)| \leq C(1+ |t|^{p-1})\quad \text{and}\quad |f'(t)| \leq C(1+ |t|^{p-2}),\quad  \forall \ t \in \mathbb{R}.
		\end{equation*}
	\end{enumerate}
	\begin{enumerate}[label=($f_2$),ref=$(f_2)$]
		\item\label{f_dois} $\displaystyle f(t) t \geq  0,$ for $t \neq 0,$ $\lim_{t \rightarrow 0 }f(t)/t = 0$ and $\lim_{|t| \rightarrow \infty }f(t)/t = + \infty.$
	\end{enumerate}	
	\begin{enumerate}[label=($f_3$),ref=$(f_3)$]
		\item\label{f_tres} The function $t \mapsto f'(t) + (1-q) f(t)/t$ is increasing for $t>0$ and decreasing for $t<0.$
	\end{enumerate}	
	\begin{enumerate}[label=($f_4$),ref=$(f_4)$]
		\item\label{f_quatro} It folds $f'(t)t^2- f(t)t > 0,$ for all $t \in \mathbb{R} \setminus \{0\}.$
	\end{enumerate}
	Throughout our work we suppose the following hypothesis for the function $a$:
	\begin{enumerate}[label=($A_1$),ref=$(A_1)$]
		\item \label{A1_weight} $a\in L^{\alpha_0} (\mathbb{R}^N),$ with $\alpha_0 = (p/q)' = p/(p-q).$ Moreover, $a(x) \geq 0$ a.e. in $\mathbb{R}^N$ and $a\not \equiv 0.$
	\end{enumerate}

	\subsection{Main Results}
	Next we suppose that the above conditions \ref{V_sirakov1}--\ref{V_sirakov3}, \ref{B_weight}, \ref{f_um}--\ref{f_quatro} and \ref{A1_weight} hold.
	\begin{theorem}\label{th1}
		There exists $\lambda^* \in (0, \infty)$ such the Problem \eqref{P} admits at least one ground state solution $u_\lambda \in X,$ for any $0<\lambda <\lambda^*$. Furthermore, $J_\lambda(u_\lambda) < 0$ and $J''_\lambda(u_\lambda)(u_\lambda,u_\lambda ) > 0.$
	\end{theorem}
	
	\begin{theorem}\label{th2}
		There exists $0 < \lambda_* <\lambda^* <\infty$ such the Problem \eqref{P} admits at least one bound state solution $v_\lambda \in X,$ for any $0<\lambda <\lambda^*.$ Furthermore, $J''_\lambda(v_\lambda)(v_\lambda,v_\lambda) < 0$ and
		\begin{enumerate}[label=\roman*)]
			\item $J_\lambda(v_\lambda) > 0,$ whenever $\lambda \in (0, \lambda_*);$
			\item $J_\lambda(v_\lambda) = 0,$ whenever $\lambda =\lambda_*;$
			\item $J_\lambda(v_\lambda ) < 0,$ whenever $\lambda \in (\lambda_*, \lambda^*);$
		\end{enumerate}
	\end{theorem}
	\begin{corollary}\label{coro}
		Problem \eqref{P} admits at least two nontrivial solutions for each $0 < \lambda < \lambda^*$.
	\end{corollary}

	\subsection{Remarks on the assumptions} Here we give some helpful comments concerning our assumptions.
	\begin{enumerate}[label=\roman*)]
		\item $\nu_\theta (\Omega )$ and $\mathcal{M}_\theta (\Omega)$ in \ref{V_sirakov3} are well defined for open domains with smooth boundaries. Details of this fact can be viewed in \cite{brasco-ariel2019}. Recently, condition \ref{V_sirakov3} appeared in \cite{desouza2020}.
		\item By \ref{f_um} and \ref{f_dois} it is clear that for any $\varepsilon >0$ there is $C_\varepsilon >0$ such that
		\begin{equation}\label{condicao_crescimento}
			|f(t)| \leq \varepsilon |t| + C_\varepsilon |t|^{p-1}\text{ and } |f'(t)| \leq \varepsilon + C_\varepsilon |t|^{p-2},\quad \forall \, t\in \mathbb{R}.
		\end{equation}
		\item Clearly, under our hypotheses, $\lim _{t \rightarrow \infty } f(t)/t =+ \infty,$ implies $\lim _{t \rightarrow \infty}F(t)/t^2 =+ \infty.$
		\item Any sum of power like functions in the following form,
		\begin{equation*}
			f(t) = \sum_{i=1}^{k} |t|^{p_i -2}t, x \in \mathbb{R}^N,\quad t \in \mathbb{R},
		\end{equation*}
		satisfies hypotheses \ref{f_um}--\ref{f_quatro} where $2 < p_1 < p_2 < \dots < p_k < 2^*_s$. Moreover, $f$ is not homogeneous: It does not satisfy $f( \alpha t) \neq \alpha^r f(t)$ for each $r > 0$ and for some $\alpha > 0$. 
		\item Assumption \ref{f_tres} is different than Ambrosetti-Rabinowitz condition (in short $(AR)$ condition). More precisely, a function $f: \mathbb{R} \to \mathbb{R}$ is said to have the $(AR)$ condition, when there exists $\theta > 2$ such that $0 < \theta F(t) \leq f(t)t ,$ for all $t \in \mathbb{R},$ where $F(t) = \int _0 ^t f(\tau ) \dtau.$ Consider the non powerlike function 
		\begin{equation*}
			f(t) = t \ln (1 + |t|), \quad t \in \mathbb{R}.
		\end{equation*}
		In this case $f$ verifies \ref{f_um}--\ref{f_quatro}, for some $p \in (2, 2^*_s),$ but $(AR)$ condition is not satisfied. In fact,
		\begin{equation*}
			F(t) = \frac{t^2}{2} \ln(1 + t) - \frac{1}{4} t^2 - \frac{1}{2}t - \frac{\ln (1 + t)}{2}, \quad  t \geq 0.
		\end{equation*} 
		Moreover, for each $\theta > 2$,
		\begin{equation*}
			f(t)t - \theta F(t) = \frac{2-\theta}{2} t^2 \ln(1 + t) +  \frac{\theta }{2}t^2 + \frac{\theta}{2} t + \frac{\theta}{2} \ln(1 + t) \to - \infty,  \mbox{ as }  t \to \infty.
		\end{equation*}
		On the other hand, we have
		\begin{equation*}
			H(t) = f'(t) + (1 - q) \frac{f(t)}{t} = (2 - q) \ln (1 + |t|) + \frac{|t|}{1 + |t|}.
		\end{equation*}
		In particular, $\lim_{|t| \to \infty} H(t) =  +\infty$. Assumptions \ref{f_um}--\ref{f_quatro} allow us to consider a huge class of nonlinearities. 
		\item Hypothesis \ref{f_tres} implies that the function $G(t) := f(t)/t - q F(t)/t^2$ is increasing for $t > 0,$ and decreasing for $t < 0$, see Proposition \ref{crucial} in Sect. \ref{s_quocient}. In particular, $f(t)t \geq q F(t),$ $ t \in \mathbb{R}.$
	\end{enumerate}
	
	\begin{remark}
		Up to our knowledge, the results presented here are new even for the local case $s=1.$ More precisely, the arguments made to prove Theorems \ref{th1} and \ref{th2} still hold true for the problem
		\begin{equation*}
			\left\{
			\begin{aligned}
				&-\Delta  u + V(x) u = \lambda a(x) |u|^{q-2}u + b(x)f(u)\quad \text{in}\quad \mathbb{R}^N,\\
				&u \in H^1(\mathbb{R}^N),
			\end{aligned}
			\right.
		\end{equation*}
		with obvious modifications in Sections \ref{s_preli} and \ref{s_var}. However, by taking $V=0,$ hypothesis \ref{B_weight} is not sufficient anymore to prove Theorems \ref{th1} and \ref{th2} for the following elliptic problems:
		\begin{equation}\label{achei1}
			\left\{
			\begin{aligned}
				&(-\Delta)^s  u  = \lambda a(x) |u|^{q-2}u + b(x)f(u)\quad \text{in}\quad \Omega,\\
				&u  = 0 \quad \text{on} \quad \mathbb{R}^N \setminus \Omega,
			\end{aligned}
			\right.
		\end{equation}
		and 
		\begin{equation}\label{achei2}
			\left\{
			\begin{aligned}
				&-\Delta  u  = \lambda a(x) |u|^{q-2}u + b(x)f(u)\quad \text{in}\quad \Omega,\\
				&u  = 0 \quad \text{on} \quad \partial \Omega,
			\end{aligned}
			\right.
		\end{equation}
		where $\Omega \subset \mathbb{R}^N$ is a bounded smooth domain with $N>2s$ and $s \in (0, 1)$. In these cases, in order to consider $b$ as an unbounded weight and to carry on our arguments, one have to require at least $b \in L^\beta (\Omega),$ for a suitable $1 < \beta < \infty.$ In the same way, we can consider an unbounded weight $a \in L^{\alpha}(\Omega)$ for some suitable $1 < \alpha < \infty$. For example, we can put $\alpha = 2/(2 -q)$ and $\beta = (2^*_s - \epsilon)/(2^*_s - \epsilon - p)$ where $\epsilon > 0$ is small enough. Under these conditions, by using the  H\"{o}lder inequality, the associated energy functional is in $C^1$ class. Furthermore, by using the same arguments explored in the present work together with the fact that embeddings of the Sobolev spaces $H^s(\Omega)$ into the Lebesgue spaces $L^r(\Omega)$ are compact for each $r \in [1,2^*_s), 2^*_s= 2 N/(N - 2s)$, the conclusions our main results (Theorems \ref{th1}, \ref{th2} and Corollary \ref{coro}) are also true for Problems \eqref{achei1} and \eqref{achei2}. For the local case $s = 1$ we can use the Sobolev space $H^1_0(\Omega)$ instead of $H^s(\Omega)$. In fact, the general function $f$ is not homegeneous and the weights $a$ and $b$ can be unbounded in $\Omega$.
	\end{remark}

	\begin{remark}\label{r_aha}
		Assume \ref{f_um}--\ref{f_quatro}. Define the function $H(t) := f'(t) + (1-q) f(t)/t$ for each $t \neq 0$. We have $\lim_{|t| \to \infty} H(t) =+  \infty$. Indeed,
		\begin{equation*}
			\frac{d}{\dtau} \left[ \frac{f(\tau)}{|\tau|^{q-2} \tau}  \right] = \frac{f'(\tau) + (1 - q)f(\tau)/\tau}{|\tau|^{q-2}\tau} = \frac{H(\tau)}{|\tau|^{q-2}\tau}.
		\end{equation*}
		The last assertion together implies 
		\begin{equation*}
			\frac{f(t)}{t^{q-1}} - \frac{f(t_0)}{t_0^{q-1}} = \int_{t_0}^{t} \frac{H(\tau)}{\tau^{q-1}} \dtau 
		\end{equation*} 
		with $0 < t_0 < t$. In view of hypothesis \ref{f_tres} we infer
		\begin{equation*}
			\frac{f(t)}{t^{q-1}} - \frac{f(t_0)}{t_0^{q-1}} \leq H(t) \int_{t_0}^{t} \frac{1}{\tau^{q-1}} \dtau = H(t) \left( \frac{t^{2-q}}{2-q} - \frac{t_{0}^{2-q}}{2-q} \right)	\leq  H(t) \frac{t^{2-q}}{2-q}.	
		\end{equation*}
		As a consequence,
		\begin{equation*}
			\frac{f(t)}{t} \leq \frac{f(t_0)}{t_0^{q-1}} \frac{1}{t^{2-q}} + \frac{H(t)}{2 - q},\ 0 < t_0 < t.
		\end{equation*}
		Now, by using hypothesis \ref{f_dois} and the last estimate, we deduce that $\lim_{t \to \infty} H(t) =  +\infty.$ Furthermore, by taking $t < t_0 < 0,$ the same argument above leads to $\lim_{t \to -\infty} H(t) =  +\infty.$
	\end{remark}
	\noindent \textbf{Notation:} In this paper, we use the following
	notations:
	\begin{itemize}
		\item  The usual norms in $L^{p}(\mathbb{R}^N)$ are denoted by $\|\cdot \| _p;$
		\item $B_R(x_0)$ is the $N$-ball of radius $R$ and center $x_0;$ $B_R:=B_R(0);$
		\item  $C_i$ denotes (possible different) any positive constant;
		\item $\mathcal{X}_A$ is the characteristic function of the set $A \subset \mathbb{R}^N;$
		\item $A^c =\mathbb{R}^N \setminus A,$ for $A \subset \mathbb{R}^N;$ 
		\item $u^+= \max\{u,0  \}$ and $u^- = \max\{-u,0  \};$
		\item $|A|$ is the Lebesgue measure of the measurable set $A \subset \mathbb{R}^N;$
	\end{itemize}
	\section{Preliminaries}\label{s_preli}
	For $\Omega \subset \mathbb{R}^N$ open set and  $0<s<1,$ the inhomogeneous fractional Sobolev space is defined as
	\begin{equation*}
		H ^s (\Omega)=\left\lbrace u \in L^2 (\Omega) :  \int _{\Omega} \int _{\Omega} \frac{\left| u(x) - u(y) \right|^2}{|x-y|^{N + 2s}}\dxdy <\infty \right\rbrace,
	\end{equation*}
	with the norm
	\begin{equation*}
		\| u \| _{H ^s (\Omega )} ^2 := \int _{\Omega} u^2\dx + \int _{\Omega} \int _{\Omega} \frac{\left| u(x) - u(y) \right|^2}{|x-y|^{N + 2s}}\dxdy .
	\end{equation*}
	We denote $\| u \| = \| u \| _{H ^s ( \mathbb{R}^N )} .$ When $0<s<1,$ by \cite[Proposition 3.4]{hitchhiker},
	\begin{equation*}
		\int _{\mathbb{R}^N} |\xi |^{2s} |\mathscr{F} u |^2 \dxi = \frac{C(N,s)}{2} \int _{\mathbb{R}^N} \int _{\mathbb{R}^N} \frac{ |u(x) - u(y) |^2}{|x-y|^{N + 2s}}\dxdy,\quad\forall \ u \in H^{s} (\mathbb{R}^N),
	\end{equation*}
	for some positive constant $C(N,s).$  Thus, when $\Omega = \mathbb{R}^N,$ we have
	\begin{equation*}
		H ^s (\mathbb{R} ^N) = \left\lbrace u \in L^2 (\mathbb{R}^N) : \ | \cdot | ^s \mathscr{F} u \in L^2 (\mathbb{R} ^N) \right\rbrace =\left\lbrace u \in L^2 (\mathbb{R}^N) : (-\Delta )^{s/2} u \in  L^2 (\mathbb{R} ^N) \right\rbrace.
	\end{equation*}
	Moreover, we have the continuous embedding
	\begin{equation*}
		H^s (\Omega) \hookrightarrow L^p(\Omega),\quad 2 \leq p \leq 2_s ^\ast,\quad\text{for}\quad 0 < s < N/2,
	\end{equation*}
	and the following compact embedding (see \cite[Section 7]{hitchhiker}),
	\begin{equation*}
		H^s(\mathbb{R}^N)\hookrightarrow L _{\loca}^p (\mathbb{R}^N),\quad 1 \leq p < 2_s ^\ast,\quad\text{for}\quad 0 < s < 1.
	\end{equation*}
	Consequently, every bounded sequence in $H^s(\mathbb{R}^N)$ has a subsequence that converges almost everywhere and weakly in $L ^p (\mathbb{R}^N),$ for $2 \leq p < 2_s ^\ast.$
	\section{Variational Settings}\label{s_var}
	We denote the space $H^s _V(\mathbb{R}^N)$ as the completion of $C_0 ^\infty (\mathbb{R}^N)$ with respect to the norm
	\begin{equation*}
		\| u \| _V = \left( \int _{\mathbb{R}^N} |\xi |^{2s} |\mathscr{F} u |^2 \dxi +\int _{\mathbb{R}^N}  V(x)|u|^2 \dx \right) ^{1/2}.
	\end{equation*}
	One can see that $H^s _V(\mathbb{R}^N)$ is a Hilbert space by taking the inner product
	\begin{equation*}
		(u,v)_V = \int _{\mathbb{R}^N} |\xi |^{2s} \mathscr{F} u  \mathscr{F} v \dxi +\int _{\mathbb{R}^N}  V(x)uv \dx,\quad u,\ v \in H^s _V(\mathbb{R}^N).
	\end{equation*}
	\begin{lemmaletter}\label{l_imersao}\cite[Proposition 6.1]{doo-ferraz2018}
		Assuming \ref{V_sirakov1} and \ref{V_sirakov2}, the space $H^s _V(\mathbb{R}^N)$ is well defined and continuously embedded in $H^s(\mathbb{R}^N).$ More precisely, there is $C>0$ such that
		\begin{equation}\label{imersao}
			\| u \| \leq C \| u \|_V,\ \forall \, u \in H^s _V(\mathbb{R}^N).
		\end{equation}
		Moreover,
		\begin{equation*}
			H^s _V(\mathbb{R}^N) \subset \left\{ u \in H^s (\mathbb{R}^N) : \int _{\mathbb{R}^N} V(x) |u|^2 \dx < + \infty \right\}.
		\end{equation*}
	\end{lemmaletter}
	\begin{lemma}\label{l_lemaini}
		If \ref{V_sirakov3} holds, then $\lim_{R \rightarrow  \infty}\nu_{\theta}(\mathbb{R}^N \setminus \overline{B}_R) =  \infty,$ for any $2\leq \theta < 2_s^\ast.$
	\end{lemma}
	\begin{proof}
		Following \cite[Lemma 2.4]{desouza2020} or \cite[Lemma 2.2]{sirakov2000}, taking $2\leq \theta _1 < \theta _2 < 2^\ast _s$ and $2<\theta _3 < \theta _4 < 2^\ast _s,$ we have
		\begin{equation*}
			\nu _{\theta _2}(\Omega)\geq C_1 (\nu _{\theta_1} (\Omega))^{\alpha _1}	\quad \text{and}\quad\nu _{\theta _3}(\Omega)\geq C_2 (\nu _{\theta_4} (\Omega))^{\alpha _2},
		\end{equation*}
		where $\Omega $ is an open domain with smooth boundary, $C_1,$ $C_2>0,$ $\alpha _1,$ $\alpha_2 \in (0,1)$ are suitable constants that does not depends on $\Omega.$ Consequently, the desired convergence follows by choosing $\Omega = \mathbb{R}^N\setminus \overline{B_R}.$
	\end{proof}
	For $1\leq \theta < \infty$ and a measurable function $K:\mathbb{R}^N \rightarrow \mathbb{R}$ let us consider the weighted Lebesgue space
	\begin{equation*}
		L^\theta _K(\mathbb{R}^N) = \left\lbrace u:\mathbb{R}^N \rightarrow \mathbb{R}: u\text{ is measurable and }\int _{\mathbb{R}^N} K(x) |u|^\theta\dx <  \infty \right\rbrace,
	\end{equation*}
	endowed with the natural norm
	\begin{equation*}
		\| u \|_{\theta,K} = \left(  \int _{\mathbb{R}^N } K(x) |u|^\theta \dx \right)^{1/\theta }.
	\end{equation*}
	\begin{proposition}\label{p_imersaopeso}
		Supposing \ref{V_sirakov1}, \ref{V_sirakov2} and \ref{B_weight}, the space $H^s _V(\mathbb{R}^N)$ is continuously embedded in $L_b^\theta (\mathbb{R}^N)$ for any $2 \leq \theta < 2^\ast _s.$ If in addition \ref{V_sirakov3} holds, then this embedding is also compact, as well the embedding of $H^s _V(\mathbb{R}^N)$ in $L ^\theta(\mathbb{R}^N).$
	\end{proposition}
	\begin{proof}
		Let $u \in H^s(\mathbb{R}^N)$ and denote $m_b (x) = \max \{1, b(x)  \}.$ By \ref{B_weight}, for any $\alpha  > 1,$ we have
		\begin{align*}
			\int _{\mathbb{R}^N}b(x)|u|^{\theta} \dx & \leq \max_{|x|\leq R_0} \left(   m_b (x) \right)   \int _{B_{R_0}} |u|^\theta \dx + \int _{B^c_{R_0} } C_0 (1+ (V_+(x))^{1/\alpha } )|u|^\theta \dx \\
			&\leq  \max_{|x|\leq R_0} \left( \left(   m_b (x) +C_0\right) \right) \| u \|^{\theta } _{\theta } + C_0\int _{\mathbb{R}^N } (V^+ (x))^{1/\alpha} |u|^{\theta } \dx.
		\end{align*}
		On the other hand, H\"{o}lder inequality implies,
		\begin{equation*}
			\int _{\mathbb{R}^N } (V^+ (x))^{1/\alpha} |u|^{\theta } \dx \leq \left( \int _{\mathbb{R}^N}  V^+(x) u^2  \dx \right)^{1/\alpha} \left(  \int _{\mathbb{R}^N } |u|^{\frac{ \alpha \theta -2}{\alpha -1}}\dx  \right)^{(\alpha-1) / \alpha}.
		\end{equation*}
		Furthermore, by \ref{V_sirakov1} and \ref{V_sirakov2}, we get
		\begin{align*}
			\int _{\mathbb{R}^N } V^+(x) u^2  \dx &\leq \| u \|_V^2 + \int _{\mathbb{R}^N }V^{-  }(x)u^2 \dx \\
			&\leq \| u \|_V^2 + B\int _{\mathbb{R}^N } u^2 \dx \leq \left( 1 + \frac{B}{\kappa_0}  \right) \| u \|^2_V.
		\end{align*}
		Summing up,
		\begin{equation}\label{desigualdade_boa}
			\int _{\mathbb{R}^N}b(x)|u|^{\theta} \dx \leq  \max_{|x|\leq R_0} \left( \left(   m_b (x) +C_0\right) \right) \| u \|^{\theta } _{\theta } + C_0 \left[ \left( 1 + \frac{B}{\kappa_0}  \right) \| u \|^2_V \right]^{1/\alpha }\| u \|^{(\alpha \theta -2)/\alpha } _{\frac{ \alpha \theta -2}{\alpha -1}}.
		\end{equation}
		Since $2\leq  (\alpha \theta -2 )/(\alpha -1) < 2^\ast_s $ for suitable $\alpha >1$ depending on $\theta,$  we can apply \eqref{imersao} in \eqref{desigualdade_boa} to obtain 
		\begin{equation*}
			\|  u \| _{\theta, b} \leq C \| u \|_V, \quad \forall \, u \in H^s_V(\mathbb{R}^N).
		\end{equation*}
		Next we suppose \ref{V_sirakov3} and take any $u_n \rightharpoonup 0$ in $H_V^s(\mathbb{R}^N).$ For a given $R>0$ let us consider $\phi_R \in C^\infty (\mathbb{R}^N)$ such that $\phi _R = 1$ in $B^c_{R+1},$ and $\phi _R  = 0 $ in $B_R,$ with $\|\nabla  \phi _R \|_\infty \leq C/R.$ Following \cite[Lemma 5.3]{hitchhiker}, it is clear that 
		\begin{equation}\label{detalhe_mais}
			[\phi _R u_n]^2_s \leq \frac{C_1}{R} \| u_n \|_2^2 + C_2[u_n]_s^2,\quad \forall \, R>0,\ n \in \mathbb{N},
		\end{equation}
		which implies $\phi _R u_n\in H^s _V(\mathbb{R}^N).$ Next we point that
		\begin{equation}\label{detalhe}
			\left( \int_{\mathbb{R}^N } |V(x)|u^2_n \dx \right)_{n} \text{ is bounded.}
		\end{equation}
		In fact, 
		\begin{align*}
			\int_{\mathbb{R}^N } |V(x)|u^2_n \dx  & \leq \int_{\mathbb{R}^N } V^+ (x)u_n^2 \dx + B \int_{\mathbb{R}^N } u_n ^2 \dx, \\
			& \leq  \int_{\mathbb{R}^N } V (x)u_n^2 \dx + 2B\int_{\mathbb{R}^N } |u_n|^2 \dx,
		\end{align*}
		and since $(\|u_n\|_V )_n$ and $(\| u_n \|_2)_n$ are bounded, we have \eqref{detalhe}. From \eqref{detalhe_mais} and \eqref{detalhe} we can conclude that $\|\phi _R u_n\|_V \leq C,$ for some constant which does not depend on $n$ and $R>1.$ Now we use the function $\phi _R u_n \| \phi _R u_n \|_\theta ^{-1}$ in the definition of $ \nu _\theta (\mathbb{R}^N \setminus \overline{B}_R)$ to get
		\begin{equation*}
			\| \phi _R u_n \|^2 _\theta \leq \left( \frac{1}{\nu _\theta (\mathbb{R}^N \setminus \overline{B}_R)} \right) \| \phi _R u_n\|^2_V.
		\end{equation*}
		Nonetheless, since the embedding 
		\begin{equation*}
			H_V^s(\mathbb{R}^N)	  \hookrightarrow H^s (\mathbb{R}^N ) \hookrightarrow L ^\beta _{\loca}(\mathbb{R}^N),
		\end{equation*}
		is compact for any $2\leq \beta < 2_s ^\ast$ (see \cite[Corollary 7.2]{hitchhiker}), clearly $u_n \rightarrow 0$ in $L ^\theta (B_{R+1}).$ Thus,
		\begin{equation*}
			\limsup _n \| u_n \|_\theta  \leq \limsup_n \|(1 - \phi_R)u_n \|_\theta + \limsup _n \| \phi_R u_n \| _\theta \leq \frac{C}{ \left(\nu _\theta (\mathbb{R}^N \setminus \overline{B}_R) \right)^{1/2}}, \quad \forall \, R>0.
		\end{equation*}
		This allow us to use Lemma \ref{l_lemaini} in order to get that $u_n \rightarrow 0$ in $L ^\theta (\mathbb{R}^N),$ for any $2\leq \theta < 2_s^\ast.$ Compactness of the embedding $H^s _V(\mathbb{R}^N) \hookrightarrow L^\theta _b(\mathbb{R}^N)$ follows by taking $u=u_n$ in \eqref{desigualdade_boa}.
	\end{proof}
	\begin{corollary}\label{c_compactA}
		In addition to the hypotheses of Proposition \ref{p_imersaopeso}, assume \ref{A1_weight}. Then $H^s _V(\mathbb{R}^N)$ is compact embedded in  $L^q_a(\mathbb{R}^N).$
	\end{corollary}
	\begin{proof}
		H\"{o}lder's inequality implies in $\| u \|_{q,a}^q \leq \| a(x)\|_{\alpha _0} \| u \|_p^q,$ for any $u \in H^s _V (\mathbb{R}^N).$ In particular, if $u_k \rightharpoonup 0$ in $H^s _V(\mathbb{R}^N),$ since $2<p<2^\ast_s,$ by Proposition \ref{p_imersaopeso} one have $\| u_k\|_{q,a} \rightarrow 0.$
	\end{proof}
	\begin{proposition}
		Suppose \ref{V_sirakov1}, \ref{V_sirakov2}, \ref{B_weight}, \ref{f_um}, \ref{f_dois} and \ref{A1_weight}. The functional $J_\lambda : H^s _V(\mathbb{R}^N) \rightarrow \mathbb{R}$ given by
		\begin{equation*}
			J_\lambda (u) = \frac{1}{2}\| u \|_V^2 - \frac{\lambda }{q}\int _{\mathbb{R}^N}a(x)|u|^{q} \dx - \int _{\mathbb{R}^N } b(x)F(u) \dx,
		\end{equation*}
		where $F(t) = \int _0 ^t f(\tau ) \dtau $ is well defined and $C^1$ with
		\begin{equation*}
			J'_\lambda (u) \varphi  = (u,\varphi )_V - \lambda \int _{\mathbb{R}^N} a(x)|u|^{q-2}u \varphi \dx  - \int _{\mathbb{R}^N} b(x) f(u) \varphi \dx ,\quad u,\ \varphi \in H^s_V(\mathbb{R}^N).
		\end{equation*}
	\end{proposition}
	\begin{proof}
		Clearly, by H\"{o}lder inequality and \ref{A1_weight},
		\begin{equation*}
			\int _{\mathbb{R}^N} a(x) |u| ^q \dx \leq \| a(x) \| _{\alpha _0} \| u \|^q _p,\quad \forall \, u \in H^s _V(\mathbb{R}^N).
		\end{equation*}
		By growth condition \eqref{condicao_crescimento} (\ref{f_um} and \ref{f_dois}) we have
		\begin{equation*}
			\int _{\mathbb{R}^N } b(x)F(u) \leq \varepsilon \| u \|^2_{2,b} + C_\varepsilon \| u \|^p_{p,b},\quad \forall \, u \in H^s _V(\mathbb{R}^N).
		\end{equation*}
		From Lemma \ref{l_imersao}, Proposition \ref{p_imersaopeso} and \eqref{desigualdade_boa} the proof now follows standard arguments.
	\end{proof}
	\section{The nonlinear Rayleigh quotient method}\label{s_quocient}
	In this section we follow the methods of \cite{yavdat2017}, where for simplicity we denote $X=H^s_V(\mathbb{R}^N).$ Next we always assume conditions \ref{V_sirakov1}--\ref{V_sirakov3}, \ref{B_weight}, \ref{f_um}--\ref{f_quatro} and \ref{A1_weight}. We start by introducing the Nehari manifold related to \eqref{P} by
	\begin{equation*}
		\mathcal{N}_\lambda = \left\lbrace   u \in X \setminus \{ 0 \} : J'_\lambda (u) u = 0  \right\rbrace  = \left\lbrace   u \in X\setminus \{ 0 \} :   \lambda  \|  u \|^q_{q,a} =\| u \|^2_V - \int_{\mathbb{R}^N} b(x)f(u)u\dx \right\rbrace .
	\end{equation*}
	The nonlinear generalized Rayleigh quotient $R_n:X\setminus \{ 0 \} \to \mathbb{R}$ is given by
	\begin{equation*}
		R_n (u) = \frac{1}{\|  u \|^q_{q,a}} \left[ \| u \|^2_V - \int_{\mathbb{R}^N} b(x)f(u)u\dx \right].
	\end{equation*}
	Roughly speaking, the nonlinear Rayleigh quotient method applied in our framework consists on finding a suitable extremal value $0<\lambda ^\ast<\infty,$ in a such way that for any $\lambda \in (0,\lambda ^\ast),$ one have $J_\lambda ''(u) (u,u) \neq 0,$ for all $u \in X\setminus\{0\}.$ In this case, if $u_0$ is a minimizer of $\inf _{u \in \mathcal{N}_\lambda} J_\lambda (u),$ then Lagrange multiplier theorem leads to the existence of $\mu \in \mathbb{R}$ with
	\begin{equation*}
		J'(u_0) u_0 = \mu J''_\lambda (u_0) (u_0,u_0),
	\end{equation*}
	where one can deduce that $\mu = 0 $ and the existence of a ground state solution of \eqref{P} is ensured. On the other hand, $R_n \in C^1 (X:\mathbb{R})$ and
	\begin{equation*}
		R'_n(u)u = \| u \|_{q,a}^{-1} J''_\lambda (u) (u,u),\quad \forall \, u \in \mathcal{N}_\lambda.
	\end{equation*}
	This identity suggests that we first look for critical points of $R_n,$ which is made by a fine analysis of the fibering map $q_n(t) :=  R_n(tu),$ $t\geq 0,$ $u \in X \setminus \{ 0 \}.$
	\begin{lemma}\label{l_unique}
		Given $u \in X \setminus \{0 \},$ there is a unique critical point of $q_n.$ Furthermore, $\lim _{t\rightarrow 0} q_n (t ) =0,$ $q_n(t) >0$ for small $t>0$ and $\lim _{t\rightarrow \infty } q_n (t ) =-\infty.$ In particular, there is a unique maximum point $t_n=t_n(u) > 0 $ of $q_n.$ 
	\end{lemma}
	\begin{proof}
		Clearly,
		\begin{equation*}
			q_n(t)= R_n(t u ) = \frac{1}{\| u \|_{q,a}^q}\left[ t^{2-q} \| u \|_V^2 - t^{1-q}\int _{\mathbb{R}^N} b(x)f(t u ) u \dx \right] ,\quad t \geq 0.
		\end{equation*}
		Therefore $q'_n(t) =0$ if, and only if
		\begin{equation*}
			(2-q) \| u \|_V^2  = \int _{\mathbb{R}^N }b(x) \left[f'(t u) + (1-q)\frac{f(tu)}{tu}   \right]u^2 \dx.
		\end{equation*}
		Hypothesis \ref{f_tres} implies on the existence of a unique $t_n >0$ such that \eqref{unique_local} holds. In particular, there is a unique critical point of $q_n.$ Nonetheless, using \ref{f_um} and \ref{f_dois} respectively, we have
		\begin{equation*}
			\lim _{t\rightarrow 0} \frac{R_n (t u )}{t^{2-q}} = \frac{\| u \|^2_V}{\|u \|^q_{q,a}} >0\quad \text{and}\quad \lim _{t \rightarrow \infty} \frac{R_n (tu)}{t^{2-q}} = - \infty.\qedhere
		\end{equation*}
	\end{proof}
	Based on Lemma \ref{l_unique}, we define an auxiliary functional $\Lambda_n: X \setminus \{ 0 \} \rightarrow \mathbb{R}$ by $\Lambda_n(u)= R_n(t_n(u) u)$ and the first Rayleigh extremal value:
	\begin{equation*}
		\lambda ^\ast = \inf_{u \in X \setminus \{ 0 \}} \Lambda_n (u).
	\end{equation*}
	
	\begin{remark}\label{zero}
		The functional $\Lambda_n: X \setminus \{0\} \to \mathbb{R}$ is $0$--homogeneous. More precisely, we have $\Lambda_n(\alpha u ) = \Lambda _n (u),$ for all $\alpha>0$ and $u \in X\setminus \{ 0 \}.$
		In fact, for $\alpha>0$ and $u \in X\setminus \{ 0 \},$ we have
		\begin{align*}
			\Lambda_n(\alpha u ) &= \sup_{t >0}\left[ \frac{1}{\| u \|_{q,a}^q}\left( (\alpha t)^{2-q} \| u \|_V^2 - (\alpha t)^{1-q}\int _{\mathbb{R}^N} b(x)f((\alpha t) u ) u \dx \right) \right]\\
			&=\sup _{\tau >0}\left[ \frac{1}{\| u \|_{q,a}^q}\left( \tau ^{2-q} \| u \|_V^2 - \tau ^{1-q}\int _{\mathbb{R}^N} b(x)f(\tau  u ) u \dx \right) \right] = \Lambda_n( u ).\qedhere
		\end{align*}
	\end{remark}
	
	\begin{corollary}\label{c_identidade}
		For any $u \in X\setminus \{0 \},$ it holds
		\begin{equation}\label{unique_local}
			(2-q)\| t_n u \|_V^2 = \int _{\mathbb{R}^N } b(x) \left(  f'(t_n u ) (t_n u )^2+(1-q)f(t_n u )(t_n u ) \right) \dx.
		\end{equation}
		Furthermore, using \eqref{unique_local} in the definition of $\Lambda _n,$
		\begin{equation*}
			R_n(t_n u ) = \frac{1}{(2-q) \|t_n u \|^q_{q,a} } \int _{\mathbb{R}^N } b(x)\left(f'(t_n u) (t_n u)^2 - f(t_n u ) (t_n u)\right) \dx.
		\end{equation*}
	\end{corollary}
	
	In what follows we are going to prove that $\lambda ^\ast $ is attained and $\lambda ^\ast>0.$ In order to do that, we are going to first prove a technical but not less important result.
	\begin{lemma}\label{l_difzero}
		There is $\eta >0$ such that $\| t_n(u) u \|_V \geq \eta $ for all $u \in X\setminus \{ 0 \}.$
	\end{lemma}
	\begin{proof}
		Denoting $v = t_n(u)u,$ by \eqref{unique_local}, we have
		\begin{equation*}
			(2-q)\| v\|_V^2 = \int _{\mathbb{R}^N  }    b(x)  \left(  f'(v ) v^2 + (1-q) f(v) v  \right)\dx.
		\end{equation*}
		Now using \eqref{condicao_crescimento} we obtain
		\begin{equation*}
			(2-q)\| v\|_V^2 \leq 2 \varepsilon \mathcal{C}_2 \| v \| _V^2 + 2 C_\varepsilon \mathcal{C}_p \| v \|^p_V,
		\end{equation*}
		where $\mathcal{C}_2$ and $\mathcal{C}_p$ are suitable positive constants derived from the embedding $H^s _V(\mathbb{R}^N) \hookrightarrow L^\theta _b(\mathbb{R}^N),$ $2 \leq \theta < 2_s ^\ast.$ Next, choosing $\varepsilon$ small enough we conclude that $\| v \|^{p-2}_V \geq (2C_\varepsilon C_p)^{-1} \left[ (2-q)-2\varepsilon C_\varepsilon  \right].$
	\end{proof}
	Let $(v_k) \subset X \setminus \{ 0 \}$ be a sequence given by $v_k := t_n(u_k)u_k,$ with $(u_k) \subset X \setminus \{ 0  \},$ in a such way that $\Lambda _n (u_k) \rightarrow \lambda ^\ast.$

	\begin{lemma}\label{l_boundedbis}
		Assume that $\lambda \in (0, \lambda^*)$. Then the sequence $(v_k)$ is bounded in $X.$
	\end{lemma}
	\begin{proof}
		The proof follows arguing by contradiction. Assume that $\|v_k\|_V \to \infty,$ as $k \to \infty$. 	By the definition of $\Lambda_n,$ up to a subsequence, we can write $\Lambda _n (u_k) < \lambda ^\ast + 1/k$ and
		\begin{equation*}
			\| v_k \|_V^2 \leq \int _{\mathbb{R}^N}b(x) f(v_k) v_k  \dx + \left( \lambda ^\ast + \frac{1}{k}  \right)\| v_k \|_{q,a}^q.
		\end{equation*}	
		By Corollary \ref{c_identidade}, we obtain
		\begin{equation*}
			(2 -q)\|v_k\|^2_V = \int _{\mathbb{R}^N}b(x) \left( f'(v_k)v_k^2 + (1 -q) f(v_k)v_k \right)  \dx .
		\end{equation*}
		As a consequence, 
		\begin{equation*}
			(2 -q) = \int _{\mathbb{R}^N}b(x) \left[f'(v_k) + (1 -q) \frac{f(v_k)}{v_k}\right] w_k^2 \dx,
		\end{equation*}
		where $w_k = v_k/\|v_k\|_V$. Notice also that $(w_k)$ is bounded in $X,$ proving the existence of $w \in X$ such that $w_k \rightharpoonup w$ in $X,$ up to a subsequence.
		
		At this stage, we shall split the proof into cases. In the first one we assume $w \neq 0$, that is, the set $[w \neq 0] = \{ x \in \mathbb{R}^N : w(x) \neq 0\}$ has positive Lebesgue measure. Therefore, $|v_k(x)| \rightarrow \infty$ a. e. in the set $[w \neq 0],$ as $k \to \infty$. Now, by using hypothesis \ref{f_tres} (see Remark \ref{crucial}) and taking into account Fatou's Lemma, we infer that 
		\begin{align*}
			(2 -q) &= \liminf_{k \to \infty} \int _{\mathbb{R}^N}b(x) \left[f'(v_k) + (1 -q) \frac{f(v_k)}{v_k}\right] w_k^2 \dx  \\
			&\geq\int _{\mathbb{R}^N} b(x) \liminf_{k \to \infty} \left[f'(v_k) + (1 -q) \frac{f(v_k)}{v_k}\right] w_k^2 \dx  \\
			&\geq\int _{[w \neq 0]} b(x) \liminf_{k \to \infty}  \left[f'(v_k) + (1 -q) \frac{f(v_k)}{v_k}\right] w_k^2 \dx = + \infty .
		\end{align*}
		This is a contradiction proving that $w \neq 0$ is impossible. This finishes the proof for the first case.
		In the second case we shall assume that $w = 0$, that is, $w_k \rightharpoonup 0$ in $X$. It is not hard to see that 
		\begin{equation*}
			R_n(tv_k) \leq R_n(v_k) = \Lambda_n(v_k) \leq \lambda^* + 1/k, \quad t > 0.
		\end{equation*}
		Using the last assertion with $t = 1/\|v_k\|$ we obtain
		\begin{equation}\label{ale1}
			1 \leq  \int _{\mathbb{R}^N}b(x) f(w_k)w_k dx + \left(\lambda^* + 1/k\right) \|w_k\|_{q,a}^q.
		\end{equation}
		On the other hand, by using hypotheses \ref{f_um} and \ref{f_dois} together with the compact embedding of $X$ into $L^p_b(\mathbb{R}^N)$ and $L^q_a(\mathbb{R}^N)$ give in Proposition \ref{p_imersaopeso}, we know that 
		\begin{equation}\label{ale2}
			\|w_k\|_{q,a}^q \rightarrow 0\quad\text{and}\quad  \int _{\mathbb{R}^N}b(x) f(w_k)w_k \dx \rightarrow  0, \quad \text{as} \ k \rightarrow \infty. 
		\end{equation}
		As a consequence, by using \eqref{ale1} and \eqref{ale2}, we also infer
		\begin{equation*}
			1 \leq  \int _{\mathbb{R}^N}b(x) f(w_k)w_k dx + \left(\lambda^* + 1/k\right) \|w_k\|_{q,a}^q \rightarrow 0, \quad \text{as} \ k \rightarrow \infty. 
		\end{equation*} 
		This contradiction proves that $(v_k)$ is bounded in $X$.
	\end{proof}

	\begin{lemma}\label{l_liminf}
		If $v_k \rightharpoonup v$ in $X,$ then $\Lambda_n (v) \leq \liminf_{k \rightarrow \infty} \Lambda_n (v_k).$
	\end{lemma}
	\begin{proof}
		By Proposition \ref{p_imersaopeso}, clearly $v_k\rightarrow v$ in $L_b ^\theta (\mathbb{R}),$ for all $2 \leq \theta < 2_s ^\ast.$ By Corollary \ref{c_compactA}, we have $\|v_k\|_{q,a}^q\rightarrow \| v \| ^q_{q,a}.$ The growth conditions on $f$ together with the fact that $\| \cdot \|_V $ is weakly lower semicontinuous in $X$ leads to the desired conclusion.
	\end{proof}
	\begin{proposition}\label{as}
		Up to a subsequence, there is $v_0 \in X \setminus \{ 0 \}$ such that $v_k \rightharpoonup v_0$ in $X$ and $\Lambda _n (v_0) = \lambda ^\ast.$ In particular $0 < \lambda ^\ast <  +\infty.$
	\end{proposition}
	\begin{proof}
		By Lemmas \ref{l_difzero} and \ref{l_boundedbis}, up to a subsequence, $v_k \rightharpoonup v_0$ in $X,$ for some $v_0 \in X \setminus \{ 0\}.$ Lemma \ref{l_liminf} implies
		\begin{equation*}
			\lambda ^\ast \leq \Lambda_n (v_0) \leq \liminf _{k \rightarrow \infty} \Lambda _n (v_k) = \lambda ^\ast,
		\end{equation*}
		which also means that $\lambda ^\ast <  + \infty.$ Nevertheless, denoting $w_0 = t(v_0) v_0,$ hypotheses \ref{f_um}, \ref{f_quatro} and Corollary \ref{c_identidade} lead to
		\begin{equation*}
			\lambda ^\ast =  \Lambda _n (v_0) = R_n (w_0) = \frac{1}{(2-q) \|w_0 \|^q_{q,a} } \int _{\mathbb{R}^N } b(x)\left(f'(w_0) w_0^2 - f(w_0) w_0\right) \dx >0.\qedhere
		\end{equation*}
	\end{proof}
	
	\subsection{A second extremal value}
	In order to classify the signal of the energy functional $J_\lambda$ for weak solutions of \eqref{P} we consider a new nonlinear generalized Rayleigh quotient type functional, which we define  as $R_e : X \setminus \{0\} \to \mathbb{R}$ with
	\begin{equation*}
		R_e(u) = \dfrac{q}{ \|u\|_{q,a}^q} \left[\frac{\|u\|_V^2}{2} - \int_{\mathbb{R}^N} b(x)F(u) \dx \right].
	\end{equation*}
	\begin{remark}\label{cl}
		$R_e(u) = \lambda$ if, and only if $J_\lambda(u) = 0$; $R_e(u) < \lambda$ if, and only if $J_\lambda(u) < 0$; $R_e(u) > \lambda$ if, and only if $J_\lambda(u) > 0$.
	\end{remark}

	Consequently, following the same discussion made about $R_n$, the necessity of an analysis for the function $q_e(t) = R_e(tu)$ with $u \in X \setminus \{0\}$ and $t \geq 0,$ is presented. We start by noticing that $R_e \in C^1(X \setminus \{0\}, \mathbb{R}).$ Also, by using hypotheses \ref{f_um} and \ref{f_dois} (see \eqref{condicao_crescimento}) together with Fatou's Lemma we have
	\begin{equation}\label{eq_sobedesce}
		\lim_{t \to 0^+} \dfrac{q_e(t)}{t^{2 - q}} =\frac{q}{2} \frac{\| u \|^2_V}{\|u \|^q_{q,a}} >0\quad \text{and}\quad \lim_{t \to \infty} \dfrac{q_e(t)}{t^{2 - q}} = - \infty. 
	\end{equation}
	As a result, there exists $t_e(u) > 0$ such that $q_e(t_e) = \max_{t > 0} q_e(t)$. By definition,
	\begin{equation*}
		R_e(tu) = \frac{q}{\|u\|_{q,a}^q} \left[  \frac{\| u \|^2_V}{2}t^{2-q} - \int _{\mathbb{
				R}^N } b(x) F(tu)t^{-q}\dx \right], \quad u \in X \setminus \{0\}.
	\end{equation*}
	In particular, we obtain that 
	\begin{align*}
		0 	&= \dfrac{\rm d}{\dt} R_e(tu) {\Big|_{t=t_e(u)}} \\
		&= \frac{q}{\| u \|^q _{q,a}} \left[ \frac{2-q}{2}\|u\|_V^2 t_e^{1-q} - \int_{\mathbb{R}^N} b(x)\left( f(t_e u)u  t_e^{-q}   -q F(t_eu) t_e^{-q-1} \right)\dx \right] ,
	\end{align*}
	where $t_e = t_e(u).$ Hence
	\begin{align}\label{aie}
		\frac{(2 -q)}{2}\|u\|_V^2 &= \int_{\mathbb{R}^N} b(x)\left(-q t_e(u)^{-2}F(t_e(u)u) + t_e(u)^{-1}f(t_e(u)u)u \right)\dx
		\nonumber \\
		&= \int_{\mathbb{R}^N} b(x)\left[\frac{f(t_e(u)u)}{t_e(u)u}  -q \frac{F(t_e(u)u)}{(t_e(u)u)^2}  \right] u^2\dx.
	\end{align}
	The next result is used to prove that $t_e(u)$ is unique determined for each $u \in X \setminus \{0\}$.	
	\begin{lemma}\label{crucial}
		Suppose \ref{f_um}--\ref{f_tres}. The function $G: \mathbb{R} \to \mathbb{R}$ given by 
		\begin{equation*}
			G(t) = \frac{f(t)}{t} - q \frac{F(t)}{t^2},
		\end{equation*}
		is increasing for $t >  0,$ and decreasing for $t<0.$ Furthermore, $\lim_{|t| \to \infty} G(t) = + \infty.$
	\end{lemma}
	\begin{proof}
		It follows from hypothesis \ref{f_tres} that 
		\begin{equation*}
			f'(t) + (1 -q) \frac{f(t)}{t} \geq f'(\tau) + (1 -q) \frac{f(\tau)}{\tau} ,\quad t \geq \tau > 0.
		\end{equation*}
		Hence 
		\begin{equation*}
			t f'(t) + (1 -q) f(t) \geq \left[f'(\tau) + (1 -q) \frac{f(\tau)}{\tau}\right]t ,\quad  t \geq \tau > 0.
		\end{equation*}
		Now, integrating by parts using the variable $t$, we observe that 
		\begin{equation*}
			t f(t) -q F(t) \geq \frac{t^2}{2} \left[f'(\tau) + (1 -q) \frac{f(\tau)}{\tau}\right],\quad  t \geq \tau > 0.
		\end{equation*}
		The last assertion implies that
		\begin{equation*}
			\tau \left(t f(t) -q F(t)  \right)  \geq \frac{t^2}{2}  \left(\tau f'(\tau) + (1 -q) f(\tau)\right),\quad t \geq \tau > 0.
		\end{equation*}
		Thus, integrating by parts in the set $[0,\tau ]$, we obtain
		\begin{equation*}
			\frac{\tau^2}{2} \left(t f(t) -q F(t) \right) \geq \frac{t^2}{2}  \left(\tau f(\tau) -q F(\tau)\right),\quad  t \geq \tau > 0.
		\end{equation*}
		Therefore, 
		\begin{equation*}
			G(t) = \frac{f(t)}{t} -q \frac{F(t)}{t^2} \geq G(\tau) = \frac{f(\tau)}{\tau} -q \frac{F(\tau)}{\tau^2} , \quad t \geq \tau > 0.
		\end{equation*}
		The proof for the case $t \leq \tau < 0$ follows the same ideas discussed just above. Next we prove that $\lim_{|t| \to \infty} G(t) = +\infty$. In fact, we observe that 
		\begin{equation*}
			\frac{\rm d}{\dtau} \left[ \frac{F(\tau)}{|\tau|^{q-2}\tau }  \right] = \frac{f(\tau)/\tau - q F(\tau)/\tau^2}{|\tau|^{q-2}\tau } = \frac{G(\tau)}{|\tau|^{q-2}\tau }.
		\end{equation*}
		The last assertion yields
		\begin{equation*}
			\frac{F(t)}{t^{q}} - \frac{F(t_0)}{t_0^{q}} = \int_{t_0}^{t} \frac{G(\tau)}{\tau^{q-1}} \dtau ,
		\end{equation*} 
		with $0 < t_0 < t$. Now, by using the fact that $t \mapsto G(t)$ is increasing for $t > 0$, we infer that
		\begin{equation*}
			\frac{F(t)}{t^{q}} - \frac{F(t_0)}{t_0^{q}} \leq G(t) \int_{t_0}^{t} \frac{1}{\tau^{q-1}} \dtau = G(t) \left( \frac{t^{2-q}}{2-q} - \frac{t_{0}^{2-q}}{2-q} \right)	\leq  G(t) \frac{t^{2-q}}{2-q}.	
		\end{equation*}
		Hence,
		\begin{equation*}
			\frac{F(t)}{t^2} \leq \frac{F(t_0)}{t_0^{q-1}} \frac{1}{t^{2-q}} + \frac{G(t)}{2 - q},\quad  0 < t_0 < t.
		\end{equation*}
		By using hypothesis \ref{f_dois} and the last estimate, we deduce that 
		$
		\lim_{t \to \infty} G(t) =  + \infty.
		$
		Furthermore, taking $t < t_0 < 0$ and using the same arguments described above, we have $\lim_{t \to -\infty} G(t) = +\infty.$		
	\end{proof}
	
	\begin{lemma}\label{l_uniqueE}
		Let $u \in X \setminus \{0 \}.$ Then $\lim _{t\rightarrow 0} q_e (t ) =0,$ $q_e(t) >0$ for small $t>0$ and $\lim _{t\rightarrow \infty } q_e (t ) =-\infty.$ Moreover, $t_e>0$ is the unique critical point of $q_e,$ with $q_e(t_e) = \max_{t > 0} q_e(t).$
	\end{lemma}
	\begin{proof}
		Let $t_e(u) > 0$ a critical point for $t \mapsto q_e(t)$. According to \eqref{aie}, we have
		\begin{equation*}
			\frac{(2 -q)}{2}\|u\|_V^2 = \int_{\mathbb{R}^N} b(x) G(t_e(u)u) u^2\dx,
		\end{equation*}
		where $G$ is defined in Lemma \ref{crucial}. Now, taking into account that $t \mapsto G(t)$ is increasing for each $t > 0$ and decreasing for $t < 0$, the equation
		\begin{equation*}
			\frac{(2 -q)}{2} \|u\|_V^2 = \int_{\mathbb{R}^N} b(x) G(t u) u^2\dx
		\end{equation*}
		has a unique solution $t = t_e(u) > 0$. The proof follows by \eqref{eq_sobedesce}.
	\end{proof}
	The behavior of $q_e$ at infinity also describes the function $t \mapsto J_\lambda (tu),$ $t>0,$ at infinity.
	\begin{corollary}\label{c_compj}
		$J_\lambda (t u) \rightarrow - \infty,$ as $t \rightarrow \infty,$ for any $u \in X\setminus \{ 0\}.$
	\end{corollary}
	\begin{proof}
		From \ref{f_dois} and Fatou's Lemma, we infer $\lim_{t \to \infty} q_e(t) = \lim_{t \to \infty}  t^{2-q}(q_e(t)/t^{2 - q}) = - \infty$. Then there is $t_0>0$ such that
		\begin{equation*}
			J_\lambda (t u ) <-2 \frac{\lambda }{q}\| u \|_{q,a}^q t^q,\quad \text{for }t>t_0.\qedhere
		\end{equation*}
	\end{proof}
	\begin{corollary}\label{cor1}
		For each $u \in X\setminus \{0 \}$, one have
		\begin{equation}\label{ul}
			\dfrac{(2-q)}{2}\| t_e(u) u \|_V^2 = \int _{\mathbb{R}^N } b(x) \left(f(t_e(u) u )(t_e(u) u )- q F(t_e(u)u) \right) \dx.
		\end{equation}
		In particular, by using \eqref{ul}, we also obtain
		\begin{equation*}
			R_e(t_e(u) u ) = \frac{q}{(2 - q)\|t_e(u) u \|^q_{q,a} } \int _{\mathbb{R}^N } b(x)\left(  f(t_e(u) u) (t_e(u) u) - 2 F(t_e(u) u ) \right) \dx > 0.
		\end{equation*}
	\end{corollary}
	\begin{proof}
		The proof follows by the use of \eqref{aie} and the same idea of Corollary \ref{c_identidade}.
	\end{proof}
	
	\begin{remark}\label{r_lambdahomo}
		The functional $\Lambda_e$ is 0--homogeneous, that is, we have $\Lambda_e(\alpha u) = \Lambda_e(u),$ $\alpha > 0,$ $u \in X \setminus \{0\}$. The proof of this fact uses the same arguments contained in Remark \ref{zero}.
	\end{remark}
	At this stage, we consider $\Lambda_e : X \setminus \{0\} \to \mathbb{R}$ given by $\Lambda_e(u) = R_e(t_e(u)u)$. The following extremal value is defined:
	\begin{equation*}
		\lambda_* = \inf_{u \in X \setminus \{0\}} \Lambda_e(u).
	\end{equation*}
	
	\begin{proposition}\label{ed3}
		There is $\eta>0$ such that $\|t_e(u) u\|_V \geq \eta,$ for any $u \in X \setminus \{0\}$.	
	\end{proposition}
	\begin{proof}
		For each $X \setminus \{0\}$, following \ref{f_um} and \eqref{ul} we have
		\begin{equation*}
			(2-q)\| v \|_V^2 \leq 2 \varepsilon \mathcal{C}_2 \| v \| _V^2 + 2 C_\varepsilon \mathcal{C}_p \| v \|^p_V,
		\end{equation*} 
		where $v = t_e(u)u$. The last inequality together with the fact that $\varepsilon > 0$ is arbitrary imply $\|v\|_V \geq \eta > 0,$ for some $\eta > 0$.
	\end{proof}
	\begin{proposition}\label{ed1}
		Let $(u_k) \subset X \setminus \{0\}$ such that $\Lambda_e(u_k) \to \lambda_*.$ Then, the sequence given by $v_k = t_e(u_k)u_k$ is bounded in $X$.
	\end{proposition}
	\begin{proof}
		Clearly, $\Lambda_e(u_k) = R_e(v_k)$ with
		\begin{equation*}
			\lambda_* \leq \Lambda_e(v_k) < \lambda_* + \frac{1}{k}\quad\text{and}\quad \dfrac{\rm d}{{\rm d}t} R_e(tv_k) {\Big|_{t=1}} = 0.	
		\end{equation*}
		As a consequence, we obtain
		\begin{equation*}
			\frac{2-q}{2}\| v_k \|_V^2 =  \int _{\mathbb{R}^N} b(x) \left( f(v_k)v_k - q F(v_k) \right) \dx.
		\end{equation*}
		The proof follows arguing by contradiction. Assume that $\|v_k\|_V \to + \infty,$ as $k \to \infty$. We have,
		\begin{equation*}
			\frac{(2 -q)}{2} = \int _{\mathbb{R}^N}b(x) \left[\frac{f(v_k)}{v_k} -q \frac{F(v_k)}{v_k^2}\right] w_k^2 dx,
		\end{equation*}
		where $w_k = v_k/\|v_k\|_V$. Since $(w_k)$ is bounded in $X,$ there exists $w \in X$ such that $w_k \rightharpoonup w$ in $X,$ up to a subsequence.
		
		We now apply the same argument made in Lemma \ref{l_boundedbis}, splitting the proof into cases. In the first one we assume $w \neq 0$, that is, the set $[w \neq 0] = \{ x \in \mathbb{R}^N : w(x) \neq 0\}$ has positive Lebesgue measure. Therefore, $|v_k(x)| \to \infty$ a. e. in the set $[w \neq 0]$ as $k \to \infty$. Now, by \ref{f_tres} and taking into account Fatou's Lemma, we infer
		\begin{align*}
			\frac{(2 -q)}{2} &= \liminf_{k \to \infty} \int _{\mathbb{R}^N}b(x) \left[\frac{f(v_k)}{v_k} -q \frac{F(v_k)}{v_k^2}\right] w_k^2 dx \nonumber \\
			&\geq\int _{\mathbb{R}^N} b(x) \liminf_{k \to \infty} \left[\frac{f(v_k)}{v_k} -q \frac{F(v_k)}{v_k^2}\right] w_k^2 dx \nonumber \\
			&\geq\int _{[w \neq 0]} b(x) \liminf_{k \to \infty}  \left[\frac{f(v_k)}{v_k} -q \frac{F(v_k)}{v_k^2}\right] w_k^2 dx = + \infty .
		\end{align*}
		This contradiction proves that $w \neq 0$ is impossible. This finishes the proof for the first case.
		In the second case we shall assume that $w = 0$, that is, $w_k \rightharpoonup 0$ in $X$. One have
		\begin{equation*}
			R_e(tv_k) \leq R_e(v_k) = \Lambda_e(v_k) \leq \lambda_* + 1/k, \quad t > 0.
		\end{equation*}
		Using the last assertion with $t = 1/\|v_k\|$ we obtain
		\begin{equation}\label{ale13}
			\frac{1}{2} \leq  \int _{\mathbb{R}^N}b(x) F(w_k) dx + \left(\lambda_* + 1/k\right) \frac{\|w_k\|_{q,a}^q}{q}.
		\end{equation}
		On the other hand, \ref{f_um} and \ref{f_dois} together with the compact embedding of $X$ into $L^p_b(\mathbb{R}^N)$ and $L^q_a(\mathbb{R}^N),$ imply
		\begin{equation}\label{ale23}
			\|w_k\|_{q,a}^q \to 0\quad\text{and}\quad\int _{\mathbb{R}^N}b(x) F(w_k)dx \to 0, \quad \text{as} \ k \to \infty. 
		\end{equation}
		Under these conditions, \eqref{ale13} and \eqref{ale23} leads to
		\begin{equation*}
			\frac{1}{2} \leq  \int _{\mathbb{R}^N}b(x) F(w_k)dx + \left(\lambda_* + 1/k\right) \frac{\|w_k\|_{q,a}^q}{q} \to 0.
		\end{equation*} 
		This contradiction proves that $(v_k)$ is bounded in $X$. This ends the proof. 
	\end{proof}

	\begin{proposition}\label{ed2}
		$\lambda_*$ is attained, that is, there exists $v \in X \setminus \{0\}$ such that $\lambda_* = \Lambda_e(v)$. In particular, $0 < \lambda_* < \infty$.
	\end{proposition}
	\begin{proof}
		Let $(v_k) \in X$ be the minimizing sequence for $\lambda_*$ given in Proposition \ref{ed1}. Since $(v_k)$ is bounded, there exists $v \in X$ such that $v_k \rightharpoonup v$ in $X,$ up to a subsequence. Now we claim that $v \neq 0$. Assuming the claim and using the fact that the norm $\| \cdot\|_V$ is weakly lower semicontinuous, one have
		\begin{equation*}
			\lambda_* \leq \Lambda_e(v) \leq \liminf_{k \rightarrow \infty} \Lambda_e(v_k) = \lambda_*.
		\end{equation*}
		As a consequence, $\lambda_* = \Lambda_e(v) > 0$ with $\lambda_* \in (0, \infty)$ and $\lambda_*$ is attained. The proof of the claim follows arguing by contradiction, that is, $v_k \rightharpoonup 0$ in $X$. By Proposition \ref{p_imersaopeso} together with hypotheses \ref{f_um} and \ref{f_dois}, we have
		\begin{equation*}
			\int_{\mathbb{R}^N} b(x) F(v_k) dx \rightarrow 0\quad  \text{and} \quad \int_{\mathbb{R}^N} b(x) f(v_k) v_k dx \rightarrow 0, \quad \text{ as }k \rightarrow \infty.
		\end{equation*}
		On the other hand, using Proposition \ref{ed3} and \eqref{ul}, we obtain
		\begin{equation*}
			0 < \dfrac{(2 -q) \eta^2}{2} \leq \dfrac{(2-q)}{2}\| v_k\|_V^2 = \int _{\mathbb{R}^N } b(x) \left(f(v_k)v_k - q F(v_k) \right) \dx = o_k(1),
		\end{equation*}
		a contradiction, proving that $v \neq 0$. This finishes the proof.
	\end{proof}
	
	\begin{proposition}\label{ed4}
		It holds,
		\begin{equation}\label{aii}
			q_n(t) - q_e(t) = \dfrac{t}{q} q'_e(t), \quad t > 0.
		\end{equation}
		Furthermore,
		\begin{enumerate}[label=\roman*)]
			\item $q_n(t) = q_e(t)$ if, and only if $t= t_e$;
			\item $q_n(t) > q_e(t)$ if, and only if $t \in (0, t_e)$;
			\item $q_n(t) < q_e(t)$ if, and only if $t \in (t_e, \infty)$.
		\end{enumerate}
	\end{proposition}
	\begin{proof}
		The proof of \eqref{aii} follows from a straightforward calculation. The items $i)$--$iii)$ are proved directly by \eqref{aii}. 
	\end{proof}
	\begin{proposition}\label{p_desiboa}
		$t_n(u) < t_e(u)$ and $\Lambda_e(u) < \Lambda_n(u),$ $u \in X \setminus \{0\}$. In particular, $ 0 < \lambda_* < \lambda^* < \infty$.
	\end{proposition}
	\begin{proof}
		Proposition \ref{ed4} implies,
		\begin{equation}\label{desig_nec}
			\Lambda_n(u) = \max_{t > 0} q_n(t) > \max_{t \in (0, t_e(u))} q_e(t) = \sup_{t > 0} q_e(t) = \Lambda_e(u). 
		\end{equation}
		Moreover, if $t_n(u) \geq t_e(u),$ then $\Lambda_n(u) = q_n (t_n(u))\leq q_e (t_n(u)) \leq \Lambda_e(u),$ a contradiction with \eqref{desig_nec}. Consequently, 
		Proposition \ref{as}, guarantees the existence of $v \in X \setminus \{0\}$ with $\lambda^* = \Lambda_n(v).$ Then, 
		\begin{equation*}
			\lambda_* \leq \Lambda_e(v) = q_e(t_e(v)) = q_n(t_e(v)) < q_n(t_n(v)) = \max_{t > 0} q_n(t) = \Lambda_n(v) = \lambda^*.\qedhere
		\end{equation*}
	\end{proof}
	It is worthwhile to mention that $\Lambda_e$ and $\Lambda_n$ are related with the energy functional $J_\lambda $ and its derivatives.
	\begin{remark}\label{rmk1}
		Let $t > 0$ and $u \in X \setminus \{0\}.$ Then,
		\begin{enumerate}[label=\roman*)]
			\item $t u \in \mathcal{N}_\lambda $ if, and only if $R_n(tu) = \lambda;$
			\item $R_n(t u)=\lambda$ if, and only if $J'_\lambda(t u) t u = 0;$
			\item $R_n(tu) > \lambda$ if, and only if $J'_\lambda(tu) tu > 0;$
			\item $R_n(tu) < \lambda$ if, and only if $J'_\lambda(tu) tu < 0.$
		\end{enumerate}
	\end{remark}
	Additionally, from the definition of $R_n$, we have that
	\begin{equation}\label{save}
		R_n'(t u)u = q'_n(t) = \frac{1}{t} \dfrac{J''_\lambda(tu) (tu, tu)}{\|tu\|_{q,a}^q}, \quad t > 0, 
	\end{equation}
	for any $u \in X \setminus \{0\}$ with $R_n(tu) = \lambda$. Eq. \eqref{save} provides some useful properties of $q'_n(t).$
	\begin{proposition}\label{der-Rn}
		Assume that $u \in X \setminus \{0\}$ satisfies $R_n(tu) = \lambda,$ for some $t > 0$. Then,
		\begin{enumerate}[label=\roman*)]
			\item $R_n'(t u)u > 0$ if, and only if $J_\lambda ''(tu) (tu, tu) > 0;$
			\item $	R_n'(t u)u   < 0$ if, and only if $J_\lambda''(tu) (tu, tu) < 0;$
			\item $	R_n'(t u)u  = 0$ if, and only if $J_\lambda ''(tu) (tu, tu) = 0.$
		\end{enumerate}
	\end{proposition}
	Similar results hold for $q'_e(t).$ Precisely, for each $u \in X \setminus \{0\}$ such that $R_e(tu) = \lambda,$ one can prove that
	\begin{equation*}
		R_e'(tu)u = q'_e(t) =\dfrac{q}{t} \dfrac{J_\lambda '(tu) tu}{\|tu\|_{q,a}^q}, \quad t > 0. 
	\end{equation*}
	\begin{proposition} Suppose that $u \in X \setminus \{0\}$ satisfies $R_e(tu) = \lambda,$ for some $t > 0$. It holds:
		\begin{enumerate}[label=\roman*)]
			\item $R_e'(tu)u > 0$ if, and only if $J_\lambda'(tu) tu > 0;$
			\item $R_e'(tu)u  < 0$ if, and only if $J_\lambda'(tu) tu < 0;$
			\item $R_e'(tu)u  = 0$ if, and only if $J_\lambda'(tu) tu = 0$.
		\end{enumerate}
	\end{proposition}

	\section{Minimization problems over the Nehari manifold} In this section we apply the obtained results about $q_e$ to classify the sign of $J_\lambda(u),$ $u\in \mathcal{N}_\lambda.$ In order to do this, let us consider the sets
	\begin{align*}
		\mathcal{N}_\lambda^- &= \left\lbrace  u \in \mathcal{N}_\lambda : J''_\lambda(u)(u,u) < 0 \right\rbrace  ,\\
		\mathcal{N}_\lambda^0 &= \left\lbrace u \in \mathcal{N}_\lambda : J''_\lambda(u)(u,u) = 0\right\rbrace,\\
		\mathcal{N}_\lambda^+ &= \left\lbrace u \in \mathcal{N}_\lambda : J''_\lambda(u)(u,u) > 0\right\rbrace. 
	\end{align*}
	In what follows we shall study the minimization problems:
	\begin{equation}\label{minimization}
		c_{\mathcal{N}_\lambda^-} = \inf_{w \in \mathcal{N}_\lambda^-} J_{\lambda}(w)\quad \text{and}\quad c_{\mathcal{N}_\lambda^+} = \inf_{w \in \mathcal{N}_\lambda ^+} J_{\lambda}(w).
	\end{equation}
	It is important to emphasize that the functional $J_\lambda$ is only in $C^1$ class. However, the second derivative of $J_\lambda$ makes sense only for some directions $u \in X$. Namely, the function 
	\begin{equation*}
		u\mapsto J''_\lambda(u)(u,u) :=  \|u\|^2_V - \lambda (q -1) \|u\|_{q,a}^q - \int_{\mathbb{R}^N} b(x) f'(u) u^2 \dx,	
	\end{equation*}
	is well defined for each $u \in X$. Furthermore, for each $u \in \mathcal{N}_\lambda$, the last identity can be written in the following form: 
	\begin{equation}\label{nehari_neg}
		J''_\lambda(u)(u,u) =  2 \|u\|^2_V - \lambda q \|u\|_{q,a}^q - \int_{\mathbb{R}^N} b(x) \left(f'(u) u^2 + f(u)u \right)\dx.
	\end{equation}
	Nevertheless, by using the implicit function theorem, one have that $\mathcal{N}_\lambda^-$ and $\mathcal{N}_\lambda^+$ are $C^1$ manifolds in $X$. 
	
	On the other hand, it is worthwhile to recall how the fibering map given by $\gamma(t) = J_\lambda(tu),$ $t\geq 0,$ $u \in X \setminus \{0\}$ is related with the Nehari method: $tu$ belongs to the Nehari set $\mathcal{N}_\lambda$ if, and only if $\gamma'(t) = 0$. As mentioned before, it is possible to use Lagrange multiplier theorem to prove that any minimizer $u \in \mathcal{N}_\lambda^-$ or $u \in \mathcal{N}_\lambda^+$ of \eqref{minimization}, respectively, is a critical point for the functional $J_\lambda$. In order to do that and study the minimization problems \eqref{minimization}, it is crucial to analyze the second derivative of $\gamma ,$ which is given by $\gamma''(t) = J''_\lambda(tu)(tu,tu) \neq 0,$ for $u \in X \setminus \{0\}$ such way that $tu \in \mathcal{N}_\lambda^+ \cup \mathcal{N}_\lambda^-$. In other words, the sets $\mathcal{N}_\lambda^-$ and $\mathcal{N}_\lambda^+$ appear to be natural constraints in order to ensure existence of critical points $u$ for $J_\lambda$ with the properties that $J''_\lambda(u)(u,u) < 0$ or $J''_\lambda(u)(u,u) > 0,$ respectively. Next we shall consider some extra properties for the Nehari subsets $\mathcal{N}_\lambda^-$ and $\mathcal{N}_\lambda^+.$
	\begin{proposition}\label{p_evazio}
		If $\lambda \in (0, \lambda^*),$ then $\mathcal{N}_\lambda ^0 = \emptyset$.
	\end{proposition} 
	\begin{proof}
		The proof follows arguing by contradiction, assuming the existence of $u \in \mathcal{N}_\lambda ^0$ with $\lambda \in (0, \lambda^*)$. As a consequence of Proposition \ref{der-Rn}, we obtain
		\begin{equation*}
			R_n(u) = \lambda\quad \text{and}\quad \dfrac{{\rm d}}{{\rm d}t} R_n(tu) {\Big|_{t=1}} = 0.
		\end{equation*}
		In particular, by using the fact that $t \mapsto q_n(t) $ has an unique critical point, we deduce that $t_n(u) =1$. Therefore, $\lambda < \lambda^* \leq \Lambda_n(u) = \max_{t>0} q_n(t) = R_n(u) = \lambda,$ a contradiction.
	\end{proof}
	Here we make clearer the reason to call $\lambda ^\ast$ a \textit{extremal value} for the use of Rayleigh quocient method.
	\begin{remark}\label{pnevazio}
		$\lambda ^\ast = \inf \{   \lambda >0: \mathcal{N}_{\lambda} ^0 \neq \emptyset\}.$
	\end{remark}
	\begin{proof}
		By Proposition \ref{p_evazio}, it is sufficient to prove that $\mathcal{N}_{\lambda^\ast} ^0 \neq \emptyset.$ We known about the existence of $w \in X \setminus \{0\}$ such that $\Lambda_n(w) = \lambda^* ,$ see Proposition \ref{as}. Hence $\lambda^* = \Lambda_n(w) = R_n(t_n(w)w)$. Remark \ref{rmk1} yields $t_n(w) w \in \mathcal{N}_{\lambda^*}$. Furthermore, by Lemma \ref{l_unique}, the number $t_n(w)$ is unique and we can deduce that  $R_n(t w) < R_n(t_n(w) w) = \Lambda_n(w)  = \lambda^*,$ for each $t \neq t_n(w).$ Particularly, the equation $R_n(t w) = \lambda^*$ has a unique solution $t = t_n(w)$. 
		Using the last assertion we infer that 
		\begin{equation*}
			R_n(t_n(w)w) = \lambda^\ast  \quad \text{and}\quad \dfrac{{\rm d}}{{\rm d}t} R_n(tw) {\Big|_{t=t_n(w)}} = 0.
		\end{equation*}
		Consequently, by Proposition \ref{der-Rn}, we conclude $t_n(w) w \in \mathcal{N}_{\lambda^*}^0$. Therefore, $\mathcal{N}_{\lambda^*}^0 \neq \emptyset$.
	\end{proof}
	
	\begin{remark}
		In what follows we point out some important facts about the topology of the sets $\mathcal{N}_\lambda^{-}$ and $\mathcal{N}_\lambda ^+,$ when $\lambda \in (0, \lambda^*),$
		\begin{enumerate}[label=\roman*)]
			\item There exists $c > 0$ such that $\|u\|_V \geq c,$ for each $u \in \mathcal{N}^-_\lambda$. In particular, $\mathcal{N}_\lambda^-$ is a closed set in $X$.
			\item $\overline{\mathcal{N}_\lambda^+} = \mathcal{N}_\lambda^+ \cup \{0\}$.
		\end{enumerate}
		\begin{proof}
			\textit{i):} Let $u \in \mathcal{N}_\lambda^-$ be a fixed function. Clearly,
			\begin{equation}\label{parteum}
				\|u\|^2_V - \lambda \|u\|_{q,a}^q - \int_{\mathbb{R}^N} b(x) f(u)u \dx = 0,
			\end{equation} 	
			and by \eqref{nehari_neg},
			\begin{equation}\label{partedois}
				2 \|u\|^2_V - \lambda q  \|u\|_{q,a}^q - \int_{\mathbb{R}^N} b(x) \left(f'(u)u^2 + f(u)u \right) \dx < 0.
			\end{equation} 	
			As a consequence, using \eqref{parteum} in \eqref{partedois},
			\begin{equation*}
				(2 -q) \|u\|^2_V <  \int_{\mathbb{R}^N} b(x)  \left( f'(u)u^2 + (1 - q)f(u)u \right) \dx.
			\end{equation*} 
			Once again by using hypotheses \ref{f_um} and \ref{f_dois} we infer that 
			\begin{equation*}
				(2 -q) \|u\|^2_V <	2 \varepsilon \mathcal{C}_2 \| v \| _V^2 + 2 C_\varepsilon \mathcal{C}_p \| v \|^p_V,
			\end{equation*} 
			holds true for each $\varepsilon > 0$. In particular, the existence of $c>0$ with $\|u\|_V \geq c$ is guaranteed. Now let $(u_k) \subset \mathcal{N}^{-}_\lambda$ such that $u_k \rightarrow u_0$ in $X.$ Taking $u = u_k$ in inequality \eqref{partedois} and passing the limit $k\rightarrow \infty $ we have the same inequality but with $u_0.$ By Proposition \ref{p_evazio} this inequality must be a strict one. This finishes the proof.
			
			\textit{ii):} It is easy to see that $\overline{\mathcal{N}_\lambda^+}  \subset \mathcal{N}_\lambda ^+ \cup \{0\}$. It remains to construct a sequence $(u_k) \subset \mathcal{N}_\lambda ^+$ such that $u_k \to 0$ in $X$. In order to do this, consider a sequence $(w_k) \subset  X$ such that $w_k \rightharpoonup 0$ in $X$ and $\|w_k\|_V = 1.$ According to Proposition \ref{compar} there exist $t^{n,+}(w_k) < t^{n,-}(w_k)$ such that $u_k = t^{n,+}(w_k)w_k \in \mathcal{N}_\lambda^+$. Nonetheless,
			\begin{equation}\label{i8}
				0 = J_\lambda'(u_k)(u_k)  = \|u_k\|_V^2 - \lambda\|u_k\|_{q,a}^q - \int_{\mathbb{R}^N} b(x) f(u_k)u_k \dx.
			\end{equation}
			and
			\begin{equation}\label{i10}
				0 < J_\lambda''(u_k)(u_k, u_k)  = 2 \|u_k\|^2_V - \lambda q \|u_k\|_{q,a}^q - \int_{\mathbb{R}^N} b(x) \left(f'(u_k) u_k^2 + f(u_k)u_k \right)\dx.
			\end{equation}
			In particular, using \eqref{i8} in \eqref{i10}, we have
			\begin{equation*}
				(2 -q) \|u_k\|^2_V > \int_{\mathbb{R}^N} b(x) \left(f'(u_k)u_k^2 + (1 -q)f(u_k)u_k\right) \dx. 
			\end{equation*}
			The same ideas discussed in the proof of Proposition \ref{l_boundedbis} ensures that the sequence $(u_k)$ is bounded in $X$. Thus, we have the existence of $t_0 \geq$ 0 such that $t^{n,+}(w_k) \to t_0,$ as $k \to \infty,$ up to a subsequence. On the other hand, in view of \eqref{i8} and the growth condition given in \eqref{condicao_crescimento} we have,
			\begin{align*}
				(t^{n,+}(w_k))^2\|w_k\|_V^2 &= \lambda (t^{n,+}(w_k))^q\|w_k\|_{q,a}^q + \int_{\mathbb{R}^N} b(x) f(t^{n,+}(w_k)w_k)(t^{n,+}(w_k)w_k) \dx \\
				&\leq  \lambda (t^{n,+}(w_k))^q\|w_k\|_{q,a}^q + \varepsilon (t^{n,+}(w_k))^2 \|w_k\|_{2,b}^2 + C_{\varepsilon} (t^{n,+}(w_k))^p \|w_k\|_{p,b}^p.
			\end{align*}
			Therefore, by using the compact embedding $H^s _V (\mathbb{R}^N) \hookrightarrow L ^\theta _b(\mathbb{R}^N)$ for $2\leq \theta < 2^\ast _s$, we obtain 
			\begin{equation*}
				( 1 - c \varepsilon)t^{n,+}(w_k)^{2-q} \|w_k\|_V^2 \leq \lambda \|w_k\|_{q,a}^q + (t^{n,+}(w_k))^{p-q}\|w_k\|_{p,b}^q	 \to 0, \,\, \mbox{as} \,\, k \to \infty,	
			\end{equation*}
			which implies $t^{n,+}(w_k) \to 0$ and $t_0 = 0$. Consequently, $u_k \to 0$ in $X.$
		\end{proof}
	\end{remark}

	\begin{proposition}\label{compar}
		Assume $\lambda \in (0, \lambda^*)$. Then for each $u\in X\setminus \{0\}$ the identity $\lambda = R_n(tu)$ has exactly two distinct roots $0<t^{n,+}(u)<t_n(u)<t^{n,-}(u),$ such that $t^{n,+}(u)u \in \mathcal{N}^+_\lambda$ and $t^{n,-}(u)u \in \mathcal{N}^-_\lambda.$ Furthermore,
		\begin{enumerate}[label=\roman*)]
			\item $t^{n,+}(u)$ and $t^{n,-}(u)$ are the unique critical points of the map $t \mapsto J_\lambda(tu)$;
			\item $t^{n,+}(u)$ is a local minimum and $t^{n,-}(u)$ is a local maximum points for the map $t \mapsto J_\lambda(tu),$ $t >0,$ respectively;
			\item The functionals $u \mapsto t^{n,+}(u)$ and $u \mapsto t^{n,-}(u)$ belong to $C^1(X \setminus \{0\}: \mathbb{R})$.
		\end{enumerate}
	\end{proposition} 
	\begin{proof} 
		\textit{i):} For each $\lambda \in (0, \lambda^*)$, we obtain $\lambda < \lambda ^\ast \leq \Lambda_n(u) = R_n(t_n(u)u),$ $u \in X \setminus \{0\}.$ Hence, by Lemma \ref{l_unique}, the identity $R_n(tu) = \lambda$ admits exactly two roots $0<t^{n,+}(u)<t_n(u)<t^{n,-}(u),$ for each $u \in X \setminus \{0\},$ and the conclusion follows by Remark \ref{rmk1}. 
		
		\textit{ii):} Notice that $R_n(t^{n,+}(u) u) = \lambda = R_n(t^{n,-}(u)u)$ is verified. Furthermore,
		\begin{equation*}
			\left.\dfrac{\rm d}{\dt} R_n(tu)\right|_{t=t^{n,+} (u)} > 0\qquad\mbox{and}\qquad \left.\dfrac{\rm d}{\dt} R_n(tu)\right|_{t=t^{n,-} (u)} < 0.
		\end{equation*} 
		Hence, by using \eqref{save} and Proposition \ref{der-Rn}, we obtain $t^{n,+} (u) u \in \mathcal{N}_\lambda^+$ and $t^{n,-} (u) u \in \mathcal{N}_\lambda^-$. In particular, $t^{n,+} (u)$ is a local minimum point for the function $t \mapsto J_\lambda(tu)$ and $t^{n,-} (u)$ is a local maximum point for $t \mapsto J_\lambda(tu)$.	
		
		\textit{iii):} Recall that $t \mapsto R_n(tu)$ has an unique critical point which is denoted by $t_n(u)$ with $u \in X \setminus \{0\}$. Now, by using the auxiliary function  $A : (0, \infty) \times (X \setminus \{0\})  \rightarrow  \mathbb{R},$ given by $A(t, u) = R_n(tu) - \lambda ,$ we obtain that $A(t, u) = 0$ if, and only if $\lambda = R_n(t u)$. Furthermore,
		\begin{equation*}
			\frac{\partial}{\partial t} A(t, u) \neq 0,
		\end{equation*}
		for $t = t^{n,+}(u)$ or $t = t^{n,-}(u),$ with $u \in X \setminus \{0\}$. In particular, it follows from implicit function theorem \cite{drabek} that $u \mapsto t^{n,+}(u)$ and $u \mapsto t^ {n,-}(u)$ are in $C^1$ class. This ends the proof. 
	\end{proof}

	Next we define the set
	\begin{equation*}
		\mathcal{E} = \{ u \in X \setminus \{0\}: J_\lambda(u) = 0 \} = \{ u \in X \setminus \{0\}: R_e(u) = \lambda\},
	\end{equation*}
	to obtain a similar result using the Rayleigh quotient $R_e$ instead of $R_n$.
	\begin{proposition}\label{compar2}
		Assume $\lambda \in (0, \lambda_*)$. Then, for each $u\in X\setminus \{0\},$ the equation $\lambda = R_e(tu),$ $t > 0,$ has exactly two distinct roots $0<t^{e,+}(u)<t_e(u)<t^{e,-}(u),$ that is, $t^{e,+}(u)u \in \mathcal{E}$ and $t^{e,-}(u)u \in \mathcal{E}.$ Furthermore, the functionals $u \mapsto t^{e,+}(u)$ and $u \mapsto t^{e,-}(u)$ belong to $C^1(X \setminus \{0\}: \mathbb{R})$.
	\end{proposition} 
	\begin{proof} 
		Follows the same lines discussed in the proof of Proposition \ref{compar}, where instead we use Lemma \ref{l_uniqueE} and the fact that $\lambda <\lambda _\ast \leq \Lambda _e (u).$
	\end{proof}
	\begin{remark}\label{r_chata}
		\begin{enumerate}[label=\roman*)]
			\item If $G\in C^1 (\mathbb{R}:\mathbb{R}),$ given in Lemma \ref{crucial}, satisfies $G'(t)>0,$ for $t>0,$ and $G'(t)<0,$ for $t<0,$ then the map $u \mapsto t_{e}(u)$ belongs to $C^1(X \setminus \{0\}: \mathbb{R})$.
			\item Analogously, if $H \in C^1 (\mathbb{R}:\mathbb{R})$ (Remark \ref{r_aha}), is such that $H'(t)>0,$ for $t>0,$ and $H'(t)<0,$ for $t<0,$ we also have that $u \mapsto t_{n}(u)$ is in $C^1(X \setminus \{0\}: \mathbb{R})$.
		\end{enumerate}
	\end{remark} 
	\begin{proof}
		\textit{i):} Consider the function $\mathcal{B} : (0, \infty) \times (X \setminus \{0\}) \to \mathbb{R}$ given by 
		\begin{equation*}
			\mathcal{B}(t, u) =	\frac{(2 -q)}{2}\|u\|_V^2  - \int_{\mathbb{R}^N} b(x)\left[\frac{f(t u)}{t u}  -q \frac{F(t u)}{(t u)^2}  \right] u^2\dx.
		\end{equation*}
		It is easy to see that $\mathcal{B}$ is in $C^1$ class and $\mathcal{B}(t,u) = 0$ if, and only if, $t = t_e(u)$ where $u \in X \setminus \{0\}$, see Lemma \ref{l_uniqueE} and Corollary \ref{cor1}. On the other hand,
		\begin{equation*}
			G(t) = \frac{f(t)}{t} - q \frac{F(t)}{t^2},
		\end{equation*}
		is increasing for $t >  0,$ and decreasing for $t<0.$ As a consequence, for $t > 0$ Fatou's lemma yields
		\begin{equation*}
			\partial_{t} \mathcal{B}(t,u) = - \dfrac{\partial}{\partial t}  \int_{\mathbb{R}^N} b(x)\left[\frac{f(t u)}{t u}  -q \frac{F(t u)}{(t u)^2}  \right] u^2\dx
			\leq - \int_{\mathbb{R}^N} b(x) \dfrac{\partial}{\partial t} \left[\frac{f(t u)}{t u}  -q \frac{F(t u)}{(t u)^2}  \right] u^2\dx < 0.
		\end{equation*}
		According to implicity function theorem \cite{drabek}, we have $t_e \in C^1 (X \setminus \{0\} : \mathbb{R}).$
		
		\textit{ii):} Now, assuming that $t \mapsto f'(t) + (1-q) f(t)/t$ is in $C^1$ class, we can also consider $\mathcal{C} : (0, \infty) \times (X \setminus \{0\}) \to \mathbb{R}$ defined by
		\begin{equation*}
			\mathcal{C}(t,u) = (2-q)\| u \|_V^2 - \int _{\mathbb{R}^N } b(x) \left[ f'(t u ) +(1-q) \frac{f(t u )}{t u } \right] u^2 \dx.
		\end{equation*}
		Like before, $\mathcal{C}$ is in $C^1$ class and $\mathcal{C}(t,u) = 0$ if, and only if, $t = t_n(u)$ (Lemma \ref{l_unique} and Corollary \ref{c_identidade}). At this point, using the same ideas employed just above, we get $t_n \in C^1( X \setminus \{0\}:\mathbb{R}).$ This ends the proof. 
	\end{proof}
	\begin{remark}\label{ajuda1}
		By the definition of $\lambda^*,$
		\begin{equation*}
			U_n := \{ u \in X \setminus \{0\} : \Lambda_n(u) > \lambda \}=X \setminus \{0\},\quad \lambda \in (0, \lambda^*).
		\end{equation*}
		Equivalently, any nonzero function $u \in X$ admits projections $0 < t^{n,+}(u) < t^{n,-}(u) < \infty,$ whenever $\lambda \in (0, \lambda^*)$ (see Proposition \ref{compar}). If $\lambda = \lambda^*,$ then $U_n \subsetneq X \setminus \{ 0 \}$ and $\mathcal{N}_{\lambda^*}^0 \neq \emptyset,$ as one can see in Remark \ref{pnevazio}. This fact give us many difficulties in order to apply the Rayleigh quotient method. For instance, Propositions \ref{p_evazio} and \ref{compar} do not hold anymore in general, and the proof of our main results fails (see Section \ref{s_main}). Summing up, it remains as an open problem to find existence of ground state solutions for Eq. \eqref{P}, when $\lambda \geq \lambda ^\ast.$ Nevertheless, the study of the case $\lambda \geq \lambda^*$ still can be fruitful: Under the conditions of Remarks \ref{zero} and \ref{r_chata}, by Corollary \ref{c_identidade}, $\Lambda_n$ is a continuous and zero homogeneuous function, which leads $U_n$ to be an open cone set. 
	\end{remark}
	
	\begin{remark}
		In a similar fashion,
		\begin{equation*}
			U_e := \{ u \in X \setminus \{0\} : \Lambda_e(u) > \lambda \} = X \setminus \{0\},\quad  \lambda \in (0, \lambda_*).
		\end{equation*}
		Thus, by Proposition \ref{compar2}, any nonzero function $u \in X$ admits projections $0 < t^{e,+}(u) < t^{e,-}(u) < \infty,$ when $\lambda \in (0, \lambda_*)$. Moreover, by Corollary \ref{cor1} and Remark \ref{r_chata}, the functional $u \mapsto \Lambda_e(u)$ is $C^1,$ and thanks to Remark \ref{r_lambdahomo}, we also get that $U_e$ is an open cone set in $X \setminus \{0\}$.
	\end{remark}
	
	\begin{remark} $X \setminus U_e \neq \emptyset,$ for any $\lambda \in [\lambda _\ast, \lambda ^\ast).$ More precisely:
		\begin{enumerate}[label=\roman*)]
			\item Let $\lambda \in (\lambda_*, \lambda^*)$ and $w \in X \setminus \{0\}$ with $\Lambda_e(w) = \lambda_*$. There is a unique $t^{n,-}(w) > 0$ with $t^{n,-}(w) w \in \mathcal{N}_\lambda^-.$ Hence
			\begin{equation*}
				R_n(t^{n,-}(w) w) = \lambda > \lambda_* = R_e(t_e(w)w) \geq R_e(t^{n,-}(w) w),
			\end{equation*}
			and $J_{\lambda}(t^{n,-}(w) w) < J_{R_e(t^{n,-}(w) w)}(t^{n,-}(w) w) < 0$. Using Remark \ref{r_lambdahomo}, $\Lambda_e(t^{n,-}(w) w) = \Lambda_e(w) = \lambda_* < \lambda$. Summing up, $t^{n,-}(w) w$ and $t_e(w)w $ belongs to $X \setminus U_e.$
			\item For $\lambda = \lambda_*,$ take $w$ as above. Additionally, since $R_n(t w) = \lambda = \lambda _\ast =R_e (t_e(w) w)$ if, and only if $t = t_e(w),$ we have $J_\lambda( t_e(w) w) = J_{R_e(t_e(w)w)} (t_e(w)w) = 0$.
		\end{enumerate}
	\end{remark}
	When $\lambda \geq \lambda_\ast,$ we have $U_e \subsetneq X \setminus \{ 0 \}.$ In what follows we establish the behavior of $u \mapsto \Lambda_e(u)$ together with an analysis of $\lambda = R_e(tu),$ $t > 0.$
	\begin{lemma}\label{l_usar3}
		Suppose $\lambda \in [\lambda_*, \lambda^*)$ and let $u \in X \setminus \{0\}.$ Then,
		\begin{enumerate}[label=\roman*)]
			\item For $\Lambda_e(u) = \lambda,$ the equation $R_e(tu) = \lambda$ admits a unique solution $t = t_e(u)$;
			\item When $\Lambda_e(u) > \lambda,$ there are $0 < t^{e,+}(u) < t^{e,-}(u) < \infty$ such that $t^{e,+}(u)u \in \mathcal{E}$ and $t^{e,-}(u)u \in \mathcal{E};$
			\item For $\Lambda_e(u) < \lambda$, the equation $R_e(tu) = \lambda$ does not have any solution $t > 0$.
		\end{enumerate}
	\end{lemma}
	\begin{proof}
		\textit{i):} Let $u \in X \setminus \{0\}$ such that $\Lambda_e(u) = \lambda$. Because $t \mapsto R_e(tu)$ has a unique maximum point $t = t_e(u),$ we have $R_e(t u) < R_e(t_e(u)u) = \Lambda_e(u) = \lambda ,$ for any $t \neq t_e(u)$.
		
		\textit{ii):} The proof is the same of Proposition \ref{compar2}, where is enough to use the fact that $\Lambda_e(u) > \lambda$. Hence the equation $R_e(tu) = \lambda$ has exactly two roots.
		
		\textit{iii):} If $\Lambda_e(u) < \lambda,$ then $R_e(tu)  \leq \Lambda _e(u) <\lambda,$ for any $t > 0$. This ends the proof.
	\end{proof}
	\begin{proposition}\label{sinal}
		Assume $\lambda \in (0, \lambda^*)$ and fix $u\in X\setminus \{0\}.$ The following statements hold:
		\begin{enumerate}[label=\roman*)]
			\item If $\lambda \in (0, \lambda_*),$ then $t^{n,-}(u)>0$ is a global maximum point for the map $t \mapsto J_\lambda(tu),$ $t > 0.$ More precisely, $\max_{t >  0} J_\lambda(tu)  = J_\lambda(t^{n,-}(u) u)  > 0$. 
			\item For $\lambda = \lambda_*,$  the number $t^{n,-}(u)>0$ is only a local maximum point for the map $t \mapsto J_\lambda(tu),$ $t>0.$ In this case, $ \max_{t >  0} J_\lambda(tu)  = J_\lambda(t^{n,-}(u) u) = 0$.
			\item Suppose $\lambda \in (\lambda_*, \lambda^*).$ 
			\begin{enumerate}[label=\alph*)]
				\item If $\Lambda_e(u) > \lambda,$ then $t^{n,-}(u)>0$ is a global maximum point for the map $t \mapsto J_\lambda(tu),$ $t> 0,$ with $\max_{t >  0} J_\lambda(tu) = J_\lambda(t^{n,-}(u) u) > 0$.
				\item If $\Lambda_e(u) = \lambda,$ then $t^{n,-}(u)>0$ is only a local maximum point for the map $t \mapsto J_\lambda(tu),$ $t > 0.$ Moreover,  $\max_{t >  0} J_\lambda(tu) = J_\lambda(t^{n,-}(u) u) = 0.$
				\item If $\Lambda_e(u) < \lambda,$ then $t^{n,-}(u)>0$ is only a local maximum point for the map $t \mapsto J_\lambda(tu),$ $t > 0.$ Furthermore, $\max_{t >  0} J_\lambda(tu) = J_\lambda(t^{n,-}(u) u) < 0.$
			\end{enumerate}
		\end{enumerate}
	\end{proposition} 
	\begin{proof} 
		\textit{i):} Consider $\lambda \in (0, \lambda_*)$ and $v = t^{n,-}(u)u \in \mathcal{N}_\lambda ^-$. First let us prove that $t_n(u) < t_e (u) < t^{n,-}(u).$ In fact, if $t^{n,-}(u) \leq t_e (u),$ by definition $R_e(t^{n,-}(u) u ) \leq R_e(t_e (u) u ),$ and using the fact that $q_n$ is decreasing for $t>t_n(u),$ we have $R_n(t^{n,-}(u) u ) \geq R_n (t_e (u) u ).$ By Proposition \ref{ed4}, we have the following contradiction:
		\begin{equation*}
			\lambda = R_n(t^{n,-}(u) u ) \geq R_n (t_e (u) u )   = R_e(t_e(u) u ) = \Lambda _e(u) \geq \lambda _\ast.
		\end{equation*}
		Moreover, Proposition \ref{ed4} also leads to $\lambda = R_n(v) < R_e(v).$ As a consequence, by using Remark \ref{cl}, we obtain $J_{\lambda}(v) > J_{R_e(v)}(v) = 0$. Since $t^{n,-}(u)$ and $t^{n,+}(u)$ are the unique two critical points of $t \mapsto J_\lambda(tu),$ by Corollary \ref{c_compj}, we conclude that $\max_{t > 0} J_\lambda(tu)  = J_\lambda(t^{n,-}(u) u) > J_\lambda (0) = 0$.
		
		\textit{ii):} Since $\lambda = \lambda_* < \lambda^*,$ by Proposition \ref{compar}, there are $t^{n,+}(u)<t^{n,-}(u)$ solutions of $R_n(tu) = \lambda.$ On the other hand, from Proposition \ref{ed4}, the equation $R_e(t u) = \lambda= R_n(tu)$ has an unique solution $t = t_e(u).$ Because of the uniqueness given in Proposition \ref{compar}, $t_e(u) = t^{n,+}(u)$ or $t_e(u) = t^{n,-}(u).$ However, $t^{n,+}(u) < t_n(u) < t_e(u),$ which implies in $t_e(u) = t^{n,-}(u).$ The last assertion says that $J_\lambda(t^{n,-}(u)u) = J_{R_e(t^{n,-}(u)u)}(t^{n,-}(u)u) = 0.$
		
		\textit{iii)--a):} According to Lemma \ref{l_usar3} there are $0 < t^{e,+}(u) < t_e(u) < t^{e,-}(u) < \infty,$ with $t^{e,+}(u)u \in \mathcal{E}$ and $t^{e,-}(u)u \in \mathcal{E}$. We are going to prove that $t_e(u) < t^{n,-}(u) < t^{e,-}(u)$. In fact, assuming $t_e(u) \geq t^{n,-}(u)$, we have the contradiction: $\lambda < \Lambda_e(u) = R_e(t_e(u)u) = R_n(t_e(u) u) \leq R_n(t^{n,-}(u)) = \lambda$. Here we used Proposition \ref{ed4} together with the fact that $t \mapsto R_n(tu)$ is a decreasing function, for each $t > t_n(u)$. It remains to ensure that $t^{n,-}(u) <  t^{e,-}(u),$ where the proof follows again by a contradiction argument. If $t^{n,-}(u) \geq  t^{e,-}(u),$ using that $t \mapsto R_e(tu)$ is a decreasing function for each $t > t_e(u)$ and Proposition \ref{ed4}, we get $\lambda = R_e(t^{e,-}(u)u) \geq R_e(t^{n,-}(u)u) > R_n(t^{n,-}(u)u)= \lambda,$ which is impossible. Next, using Proposition \ref{ed4} again, we obtain $\lambda = R_n(t^{n,-}(u) u) < R_e( t^{n,-}(u) u)$. The last assertion implies $J_\lambda(t^{n,-}(u) u) > J_{R_e(t^{n,-}(u) u)}(t^{n,-}(u) u) = 0$ and the conclusion follows by Corollary \ref{c_compj}.

		\textit{iii)--b):} By Lemma \ref{l_usar3}, there exists a unique $t = t_e(u)$ such that $R_e(t u) = \lambda.$ Hence $R_e(t_e(u)u) = \lambda = R_n(t^{n,-}(u)u)$. Arguing as above, we have $t_e(u) = t^{n,-}(u)$ and $J_\lambda(t^{n,-}(u) u) = J_{R_e(t_e(u)u)} (t_e(u)u) = 0$.

		\textit{iii)--c):} In this case, $R_e(t^{n,-} (u) u ) \leq \Lambda _e (u) < \lambda.$ Thus $J_\lambda (t^{n,-} (u) u ) < J_{R_e(t^{n,-} (u) u )} (t^{n,-} (u) u ) = 0.$ This ends the proof.
	\end{proof}

	\begin{remark}\label{obs}
		Clearly, as a consequence of Proposition \ref{sinal}, if $\lambda \in (0, \lambda^*)$ and $u\in X\setminus \{0\},$ then
		\begin{equation*}
			\max_{t > 0} J_\lambda(t u) = \max \left\lbrace  0, J_\lambda(t^{n,-}(u)u) \right\rbrace .
		\end{equation*}
	\end{remark}

	Now we prove that any function $u \in \mathcal{N}_\lambda^+$ has a negative energy.
	
	\begin{proposition}\label{importante_c+}
		Assume $\lambda \in (0, \infty)$. Then $J_\lambda(v) < 0,$ for any $v \in \mathcal{N}^+_\lambda.$ In particular, $c_{\mathcal{N}_\lambda^+} < 0$.
	\end{proposition}
	\begin{proof}
		Using Proposition \ref{compar} and the definition of $\mathcal{N}_\lambda,$ one have $t^{n,+}(v) = 1$ and $R_n(v) = \lambda$. According to Proposition \ref{ed4}, it holds $R_n(tv) > R_e(tv),$ for each $t < t_e(v)$. In particular, since $1 = t^{n,+}(u) < t_n(u) < t_e(u)$, we have $R_e(v) < R_n(v) = \lambda$. This assertion together with Remark \ref{cl} imply $J_\lambda(v) < J_{R_e(v)}(v) = 0.$
	\end{proof}
	We use the next result to ensure that minimizing sequences for $c_{\mathcal{N}_\lambda ^+},$ $c_{\mathcal{N}_\lambda ^-}$ and $c_{\mathcal{N}_\lambda} $ are bounded.
	
	\begin{proposition}\label{coercive}
		If $\lambda \in (0,\lambda ^\ast),$ then $J_\lambda $ restricted to $\mathcal{N}_\lambda $ is coercive. More precisely, $J_\lambda (u) \rightarrow + \infty,$ as $\| u \|_V \rightarrow + \infty,$ with $u \in \mathcal{N}_\lambda .$
	\end{proposition}
	\begin{proof}
		The proof follows arguing by contradiction. Let us assume the existence of a sequence $(v_k) \subset \mathcal{N}_\lambda ,$ such that $\|v_k\|_V \to \infty,$ as $k \to \infty,$ with $J_\lambda(v_k) \leq C,$ for some $C > 0$. Consider the sequence $w_k = v_k/\|v_k\|_V$. It follows that $\|w_k\|_V = 1$ and $w_k \rightharpoonup w$ in $X,$ up to a subsequence, for some $w \in X.$ Once again, we shall split the proof into cases. In the first one we assume that $w \neq 0$, that is, the set $[w \neq 0] = \{ x \in \mathbb{R}^N : w(x) \neq 0\}$ has positive Lebesgue measure. Therefore, $|v_k(x)| \to \infty$ a.e. in the set $[w \neq 0],$ as $k \to \infty$. On the other hand, we observe that 
		\begin{align*}
			C \geq	J_\lambda (v_k) &= J_\lambda (v_k) - \frac{1}{q} J_\lambda '(v_k) v_k \nonumber \\
			&=\left(\frac{1}{2} - \frac{1}{q}\right)\| v_k \|_V^2 + \frac{1}{q}\int _{\mathbb{R}^N} b(x) \left(f(v_k)v_k - q F(v_k) \right) \dx. 
		\end{align*}
		Consequently,
		\begin{equation*}
			\int _{\mathbb{R}^N} b(x)\left(f(v_k)v_k - q F(v_k) \right)\dx \leq C_1 + \frac{2-q}{2} \| v_k \|_V^2,
		\end{equation*}
		for some $C_1 > 0$. Now, by using \ref{f_tres} (see Lemma \ref{crucial}), the last estimate and Fatou's lemma, we deduce
		\begin{align*}
			\frac{2 -q}{2} &\geq \liminf_{k \to \infty} \int _{\mathbb{R}^N} b(x) \frac{f(v_k)v_k -q F(v_k)  }{\|v_k\|_V^2} \dx  \\
			&= \liminf_{k \to \infty} \int _{\mathbb{R}^N} b(x) \left[\frac{f(v_k)}{v_k} -q  \frac{F(v_k)}{v_k^2}\right] w_k^2 \dx  \\
			&\geq\int _{\mathbb{R}^N} b(x) \liminf_{k \to \infty} \left[\frac{f(v_k)}{v_k} -q \frac{F(v_k)}{v_k^2}\right] w_k^2 \dx  \\
			&\geq \int _{[w \neq 0]} b(x) \liminf_{k \to \infty}  \left[\frac{f(v_k)}{v_k}-q \frac{F(v_k)}{v_k^2}\right] w_k^2 \dx = + \infty .
		\end{align*}
		This contradiction proves that $w \neq 0$ is impossible. In the second case we consider $w = 0$, that is, $w_k \rightharpoonup 0$ in $X$. Hence, by using \ref{f_um} and \ref{f_dois} together with the compact embedding of $X$ into $L^p_b(\mathbb{R}^N)$ and $L^q_a(\mathbb{R}^N),$ given in Proposition \ref{p_imersaopeso}, we have
		\begin{equation}\label{ale222}
			\|w_k\|_{q,a}^q \rightarrow 0,\ \int _{\mathbb{R}^N}b(x) f(w_k)w_k dx \rightarrow 0\ \text{and}\ \int _{\mathbb{R}^N}b(x) F(dw_k)dx \rightarrow 0,\ d>0, \ \text{as} \ k \to \infty. 
		\end{equation}
		On the other hand, taking $t = d/\|v_k\|_V,$ $d > 0$, we get
		\begin{align}
			J_\lambda(t v_k) &= \frac{\|d w_k\|^2_V}{2}- \frac{\lambda}{q} \|d w_k\|_{q,a}^q - \int_{\mathbb{R}^N} b(x) F(d w_k)\dx \nonumber \\
			&= \frac{d^2}{2}- \frac{d^q\lambda}{q} \| w_k\|_{q,a}^q - \int_{\mathbb{R}^N} b(x) F(d w_k) \dx.  \label{precisa}
		\end{align}
		By \eqref{ale222} and \eqref{precisa},
		\begin{equation}\label{af}
			\lim_{k \to \infty} J_\lambda(t v_k) = d^2/ 2 > 0,\quad d > 0.	
		\end{equation}
		Nevertheless, from Proposition \ref{p_evazio}, $\mathcal{N}_\lambda ^0 =\emptyset,$ and to conclude the proof is sufficient to analyze two distinct cases: $(v_k) \subset \mathcal{N}_\lambda ^-$ or $(v_k) \subset \mathcal{N}_\lambda ^+.$ For the first case, by Proposition \ref{sinal} and Remark \ref{obs}, there exists $C > 0$ such that 
		\begin{equation*}
			J_\lambda(t v_k) \leq \max_{t > 0} J_\lambda(t v_k) = \max\left\{ 0, J_\lambda(v_k)\right\} \leq C,\ \forall \, t > 0,
		\end{equation*}
		where we used the fact that $t^{n,-}(v_k) = 1$ and $J(v_k) \leq C,$ for each $k \in \mathbb{N}$. Under these conditions, we obtain
		\begin{equation*}
			J_\lambda(tv_k)  \leq C, \quad \forall \,  t > 0,\  k \in \mathbb{N},
		\end{equation*}
		which by \eqref{af}, leads to
		\begin{equation*}
			d^2/ 2  = \lim_{k \to \infty} J_\lambda(t v_k), 
		\end{equation*}
		a contradiction, since $d > 0$ is arbitrary. For the case $(v_k) \subset \mathcal{N}_\lambda ^+,$ the same argument just above can be considered to get the following contradiction
		\begin{equation}\label{aff}
			0< d^2/2 =\lim_{k \to \infty} J_\lambda(t v_k) \leq  0,\quad \forall \, d > 0,	
		\end{equation}
		where in the last inequality Proposition \ref{importante_c+} is used. Summing up, $J_\lambda \big|_{\mathcal{N}_\lambda}$ is coercive.
	\end{proof}
	In the following results, we study minimization problems over $\mathcal{N}_\lambda^-,$ $\mathcal{N}_\lambda^+$ and $\mathcal{N}_\lambda.$
	\begin{proposition}\label{N-}
		Assume $\lambda \in (0,\lambda^\ast).$ Let $(u_k) \in \mathcal{N}_\lambda ^-$ be a minimizing sequence to the functional $J_\lambda$ on the Nehari subset $\mathcal{N}_\lambda ^-$. Then, $(u_k)$ is bounded in $X$ and there exists $u \in  \mathcal{N}_\lambda^-$ such that $u_k \to u$ in $X,$ with $u \neq 0$. In particular, $c_{\mathcal{N}_\lambda^-} = J_\lambda(u)$.
	\end{proposition}
	\begin{proof}
		In view of Proposition \ref{coercive}, the sequence $(u_k)$ is bounded in $X$. Hence there exists $u \in X$ such that $u_k \rightharpoonup u$ in $X,$ up to a subsequence. Now we claim that $u \neq 0$. The proof of this follows arguing by contradiction, assuming $u = 0.$ Once again we obtain 
		\begin{equation*}
			\|u_k\|_{q,a}^q \rightarrow 0 \quad\text{and}\quad \int_{\mathbb{R}^N} b(x) f(u_k)u_k \dx \rightarrow 0,\quad \text{as } k \rightarrow \infty. 
		\end{equation*}
		On the other hand,
		\begin{equation*}
			0 < c^2 \leq \lim_{k \infty} \|u_k\|^2_V = \lim_{k \to \infty} \left[ \|u_k\|_{q,a}^q + \int_{\mathbb{R}^N} b(x) f(u_k)u_k \dx \right]= 0.  
		\end{equation*}
		This contradiction proves that $u \neq 0$. It remains to prove that $u_k \to u$ in $X$. We use a contradiction argument again, assuming $\|u\|_V < \liminf_{k \to \infty} \|u_k\|_V$. Recalling Proposition \ref{compar}, there exists a unique $t^{n,-}(u) > 0$ such that $t^{n,-}(u) u \in \mathcal{N}_\lambda ^-$. Furthermore, because the functional $u \mapsto R_n(u)$ is also weakly lower semicontinous, $R_n(tu) < \liminf_{k \to \infty} R_n(tu_k)$ holds, for each $t > 0$. Hence $R_n(tu) < R_n(tu_k)$ is satisfied for each $k$ large enough, up to a subsequence. The last assertion implies
		\begin{equation*}
			0 < t^{n,+}(u_k) < t^{n,+}(u) < t^{n,-}(u) < t^{n,-}(u_k),
		\end{equation*}
		for $k$ big enough. Nevertheless, $R_n (t u_k) >\lambda$, for any $t \in (t^{n,+}(u_k), t^{n,-}(u_k)),$ and we can infer from Remark \ref{rmk1} that the function $t \mapsto J_\lambda(tu_k)$ is increasing in the set $(t^{n,+}(u_k), t^{n,-}(u_k))$. Under those conditions, since $u \mapsto J_\lambda (u)$ is weakly lower semicontinous, we have a contradiction:
		\begin{equation*}
			c_{\mathcal{N}_\lambda ^-} \leq J_\lambda(t^{n,-}(u)u) < \liminf_{k \to \infty} J_\lambda(t^{n,-}(u)u_k) \leq \liminf_{k \to \infty} J_\lambda(t^{n,-}(u_k)u_k) = c_{\mathcal{N}_\lambda ^-},
		\end{equation*}
		where we used the fact that $J_\lambda ''(u_k)(u_k,u_k)<0$ together with Proposition \ref{compar}, to get $t^{n,-}(u_k)=1.$ Now since $X$ is a Hilbert space, the convergence $\|u_k\|_V \rightarrow \| u \|_V$ implies in $u_k \rightarrow u$ in $X.$ Particularly, $c_{\mathcal{N}_\lambda ^-} = J_\lambda(u).$ On the other hand, by Proposition \ref{compar}, $t^{n,-} : X \setminus \{0\} \to \mathbb{R}$ is in $C^1$ class and so $t^{n,-}(u) = 1$. Thus $u = t^{n,-}(u) u \in \mathcal{N}_\lambda ^{-}.$
	\end{proof}
	\begin{proposition}\label{N+}
		Suppose $\lambda \in (0, \lambda^*)$. Let $(v_k) \subset \mathcal{N}_\lambda ^+$ be a minimizing sequence to the functional $J_\lambda$ on the Nehari subset $\mathcal{N}_\lambda ^+$. Then $(v_k)$ is bounded in $X$ and there exist $v \in \mathcal{N}_\lambda ^+$ such that $v_k \to v$ in $X,$ with $v \neq 0$. In particular, $J_\lambda(v) = c_{\mathcal{N}_\lambda ^+}   < 0$.
	\end{proposition}
	\begin{proof}
		Using Proposition \ref{coercive}, the sequence $(v_k)$ is a bounded. Up to a subsequence, there exists $v \in X$ such that $v_k \rightharpoonup v$ in $X.$ Once again we claim: $v \neq 0$. Otherwise, $v_k \rightharpoonup  0$ in $X$ and by Proposition \ref{p_imersaopeso},
		\begin{equation*}
			\|v_k\|_{q,a}^q \rightarrow 0\quad\text{and}\quad \int_{\mathbb{R}^N} b(x) f(v_k)v_k \dx \rightarrow 0,\quad \text{as }k \rightarrow \infty.
		\end{equation*}
		As a consequence, $(v_k)\subset \mathcal{N}^+_\lambda $ implies $\|v_k\|_V \to 0.$ However, $c_{\mathcal{N}_\lambda ^+} = \lim_{k \to \infty} J_\lambda(v_k) = 0,$ a contradiction with Proposition \ref{importante_c+}.
		
		We are going to prove that $v_k \rightarrow v$ in $X$. The proof follows arguing by contradiction, supposing $\|v\|_V < \liminf_{k \to \infty} \|v_k\|_V$. Clearly, under these conditions, $R_n(t v) < \liminf_{k \to \infty} R_n(t v_k)$ for any $t > 0$. Nevertheless, by Proposition \ref{compar}, there are $0 < t^{n,+}(v) < t^{n,-}(v) < \infty$ such that $t^{n,+}(v) v \in \mathcal{N}_\lambda ^+$. Next we use a similar argument as made in Proposition \ref{N-}. In fact, we have
		\begin{equation}\label{re}
			0 < t^{n,+}(v_k) < t^{n,+}(v) < t^{n,-}(v) < t^{n,-}(v_k) ,
		\end{equation}
		for $k \in \mathbb{N}$ large enough. The functional $u \mapsto J'_\lambda(u)u$ is weakly lower semicontinuous, hence, $J'_\lambda(v)v < \liminf_{k \to \infty} J'_\lambda(v_k) v_k = 0$. By Remark \ref{rmk1}, $J'_\lambda(v)v <0$ implies $R_n(v) < \lambda$ and so $1 \in (0, t^{n,+}(v)) \cup (t^{n,-}(v), \infty)$. Assume $1 \in (0, t^{n,+}(v)).$ Remark \ref{rmk1} also indicates that $t \mapsto J_\lambda (tv)$ is decreasing in $(0,t^{n,+}(v)) \cup (t^{n,-}(v) ,  \infty ).$ Therefore, since $u \mapsto J_\lambda(u)$ is also weakly lower semicontinuous,
		\begin{equation*}
			c_{\mathcal{N}_\lambda ^{+}} \leq J_\lambda(t^{n,+}(v) v) \leq J_\lambda(v) < \liminf_{k \to \infty} J_\lambda(v_k) = c_{\mathcal{N}_\lambda ^+}, 
		\end{equation*}
		a contradiction. Now suppose $1 \in (t^{n,-}(v), \infty)$. Likewise $t^{n,+}(v_k) = 1.$ Thus, in view of \eqref{re} we have $1 = t^{n,+}(v_k) < t^{n,-}(v) < 1,$ a contradiction again. Summing up, $v_k \rightarrow v$ in $X$ and $c_{\mathcal{N}^+_\lambda } = J_\lambda(v)$. In particular, we have $t^{n,+}(v) = \lim_{k \to \infty} t^{n,+}(v_k) = 1$ and by Proposition \ref{compar}, $v =  t^{n,+}(v) v \in \mathcal{N}_\lambda ^+$.
	\end{proof}
	
	\begin{proposition}\label{Nehari}
		Assume $\lambda \in (0, \lambda^*)$. Let
		\begin{equation*}
			c_{\mathcal{N} _\lambda }= \inf_{w \in \mathcal{N} _\lambda } J_\lambda (w).
		\end{equation*}
		If $(u_k) \subset \mathcal{N}_\lambda$ is such that $J_\lambda (u_k) \rightarrow c_{\mathcal{N} _\lambda } ,$ then $(u_k)$ is bounded in $X$ and there exists $u \in \mathcal{N}_\lambda$ such that $u_k \rightarrow u$ in $X,$ with $u \neq 0.$ In particular, $J_\lambda(u) = c_{\mathcal{N}_\lambda}$.
	\end{proposition}
	\begin{proof}
		It follows by using the same arguments of Propositions \ref{N-} and \ref{N+} together with the fact that $c_{\mathcal{N}  _\lambda } \leq c_{\mathcal{N}_\lambda ^+}$ and $c_{\mathcal{N} _\lambda } \leq c_{\mathcal{N}_\lambda ^-}. $ Indeed, in the proof of Proposition \ref{N+}, one can see that $u_k \rightharpoonup u$ in $X,$ up to a subsequence, with $u \neq 0.$ Following the same lines, $1 \in (0, t^{n,+}(u)) \cup (t^{n,-}(u), \infty)$. If $1 \in (0, t^{n,+}(u)),$ then
		\begin{equation*}
			c_{\mathcal{N}_\lambda} \leq c_{\mathcal{N}_\lambda ^{+}} \leq J_\lambda(t^{n,+}(u) u) \leq J_\lambda(u) < \liminf_{k \to \infty} J_\lambda(u_k) = c_{\mathcal{N}_\lambda }, 
		\end{equation*}
		a contradiction. Now consider $1 \in (t^{n,-}(u), \infty)$. If there is a $t^{n,+}(u_k) = 1, $ then the inequality $t^{n,+}(u_k) < t^{n,+}(u) < t^{n,-}(u) < t^{n,-}(u_k)$ leads to a contradiction $1=t^{n,+}(u_k) < t^{n,+}(u) <1.$ Hence $t^{n,-}(u_k) = 1,$ for all $k.$ In this case, taking into account the proof of Proposition \ref{N-}, we have
		\begin{equation*}
			c_{\mathcal{N}_\lambda} \leq c_{\mathcal{N}_\lambda ^-} \leq J_\lambda(t^{n,-}(u)u) < \liminf_{k \to \infty} J_\lambda(t^{n,-}(u)u_k) \leq \liminf_{k \to \infty} J_\lambda(t^{n,-}(u_k)u_k) = c_{\mathcal{N}_\lambda },
		\end{equation*}
		a contradiction. Thus $u_k \rightarrow u$ in $X$ and $J_\lambda(u) = c_{\mathcal{N}_\lambda}.$
	\end{proof}
	Now the sign of $J_\lambda (u)$ can be described for minimizers $u \in \mathcal{N}_\lambda ^-$ of $c_{\mathcal{N}_\lambda ^-}.$
	\begin{proposition}\label{p_sinal}
		Consider $\lambda \in (0, \lambda^*)$. Let $u \in \mathcal{N}_\lambda ^-$ be a minimizer on the Nehari subset $\mathcal{N}_\lambda ^-$ obtained in Proposition \ref{N-}.
		\begin{enumerate}[label=\roman*)]
			\item If $\lambda \in (0, \lambda_*),$ then $J_\lambda(u) > 0;$
			\item $J_\lambda(u) = 0,$ for $\lambda =\lambda_*$;
			\item $J_\lambda(u) < 0,$ for each $\lambda \in (\lambda_*, \lambda^*)$;
		\end{enumerate}
	\end{proposition}
	\begin{proof}
		\textit{i):} Since $u \in \mathcal{N}_\lambda^-,$ by Proposition \ref{compar}, we have $t^{n,-}(u) = 1$. The proof of \textit{i)} and \textit{ii)} follows by Proposition \ref{sinal}--\textit{i) and \textit{ii)}}, respectively.
		
		\textit{iii):} Let $w \in X \setminus \{0\}$ such that $\lambda_* = \Lambda_e(w) = R_e(t_e(w)w).$ The equation $R_n(tw) = \lambda _\ast = R_e(tw)$ has a unique solution $t = t_e(w),$ and we can write $R_n(t_e(w) w) = R_n(t^{n,-}(w)w) = R_n(t^{n,+}(w)w) = \lambda _\ast,$ where $t^{n,-}(w)$ and $t^{n,+}(w)$ are given by taking $\lambda = \lambda _\ast$ in Proposition \ref{compar}. Since $t^{n,+}(w) < t_n(w)<t_e(w)$ and $t \mapsto R_n(tw)$ is a injection, we get that $t_e(w) = t^{n,-}(w).$ Consequently, because $\lambda > \lambda_*,$
		\begin{equation*}
			J_\lambda (u) = c_{\mathcal{N} ^- _\lambda } \leq J_\lambda (t^{n,-}(w) w ) < J_{R_e(t_e(w)w)} (t_e(w) w ) = 0.
		\end{equation*}
		This finishes the proof.
	\end{proof}
	Using the fact that $c_{\mathcal{N}_\lambda^-}$ is attained, we can compare it with $c_{\mathcal{N}^+_\lambda}.$
	\begin{proposition}\label{sacada}
		$c_{\mathcal{N}^+_\lambda} < c_{\mathcal{N}^-_\lambda},$ whenever $\lambda \in (0, \lambda^*)$.
	\end{proposition}
	\begin{proof}
		Proposition \ref{N-}, guarantees the existence of $v \in \mathcal{N}_\lambda^-$ with $c_{\mathcal{N}_\lambda^-} = J_\lambda(v).$ Because $t^{n,-}(v) =1,$ we have
		\begin{equation*}
			c_{\mathcal{N}_\lambda^+} \leq J_\lambda (t^{n,+}(v)v) < J_\lambda (t^{n,-}(v)v) = J_\lambda (v) = c_{\mathcal{N}_\lambda ^-},
		\end{equation*}
		where in the last inequality we used the fact that $t \mapsto J(tu)$ is increasing on $(t^{n,+}(u) , t^{n,-}(u))$ (see Remark \ref{rmk1} and Proposition \ref{compar}).
	\end{proof}
	\begin{remark}Let
		\begin{equation}\label{minimal}
			c_\lambda = \inf \left\lbrace J_\lambda (w) : w \in X\setminus \{0\}\text{ and }J'_\lambda(w)=0 \right\rbrace,
		\end{equation}
		the so called ground state level. According to Proposition \ref{sacada},
		\begin{equation*}
			J_\lambda(v) \geq  c_{\mathcal{N}_\lambda^-} > c_{\mathcal{N}_\lambda^+}  \geq  c_\lambda,\quad \forall \, v \in \mathcal{N}_\lambda ^-.
		\end{equation*}
		Therefore, there is no ground states of Problem \eqref{P} in $\mathcal{N}_\lambda ^-$ (see the proof of Theorem \ref{th1}).
	\end{remark}
	\section{Proof of our main results}\label{s_main}
	\begin{proof}[Proof of Theorem \ref{th1} completed] In view of Proposition \ref{Nehari} there is $u_0 \in \mathcal{N}_\lambda $ such that $c_{\mathcal{N}_\lambda } = J_\lambda(u_0) .$ If $u_0 \in \mathcal{N}_\lambda ^-,$ by Proposition \ref{sacada}, we have
		\begin{equation*}
			J_\lambda (u_0) = c_{\mathcal{N}_\lambda} \leq c_{\mathcal{N}^+_\lambda} < c_{\mathcal{N}^-_\lambda} \leq J_\lambda (u_0),
		\end{equation*}
		which is a contradiction. Because $\mathcal{N}_\lambda ^0 = \emptyset$ (see Proposition \ref{p_evazio}), we conclude that $u_0 \in \mathcal{N}_\lambda ^+.$ Consequently,
		\begin{equation*}
			J_\lambda (u_0) = c_{\mathcal{N}_\lambda } \leq c _{\mathcal{N}_\lambda ^+} \leq J_\lambda (u_0),
		\end{equation*}
		that is, $c_{\mathcal{N}_\lambda ^+} = c_{\mathcal{N}_\lambda } = J_\lambda (u_0).$ We can apply Lagrange multiplier theorem, to get some $\mu \in \mathbb{R}$ with $J'_\lambda (u_0)u_0 = \mu J''_\lambda (u_0)(u_0,u_0).$ Once again, since $\mathcal{N}_\lambda ^0 = \emptyset,$ we have $\mu = 0 $ and $u_0$ is a critical point of $J_\lambda.$ Thus $c_{\mathcal{N}_\lambda} \leq c_\lambda \leq J_\lambda (u_0) = c_{\mathcal{N}_\lambda},$ where $c_\lambda$ is given in \eqref{minimal}. In this case, we obtain 
		\begin{equation*}
			J_\lambda (u_0) = c_\lambda =c_{\mathcal{N}_\lambda}=c_{\mathcal{N}^+_\lambda}<0.	\qedhere
		\end{equation*}
	\end{proof}
	\begin{proof}[Proof of Theorem \ref{th2} completed]
		According to Proposition \ref{N-} there exists $v_0 \in \mathcal{N}_\lambda ^-$ such that 
		\begin{equation*}
			c_{\mathcal{N}_\lambda ^-}  = \inf_{w \in \mathcal{N}_\lambda ^-} J_\lambda(w)= J_\lambda(v_0).
		\end{equation*}
		Applying the Lagrange multiplier theorem we obtain the existence $\mu \in \mathbb{R}$ such that $J'_\lambda (v_0)v_0 = \mu J_\lambda '' (v_0)(v_0,v_0).$ Like above, we have $\mu =0$ and that $v_0$ is a weak solution of Problem \eqref{P}. Now the desired result follows immediately from Proposition \ref{p_sinal}. 
	\end{proof}
	\begin{proof}[Proof of Corollary \ref{coro} completed]
		The proof of Corollary \ref{coro} follows directly from Theorem \ref{th1}, Theorem \ref{th2} and the fact that $\mathcal{N}_\lambda^- \cap \mathcal{N}_\lambda^+ = \emptyset$.
	\end{proof}
	
	\section*{Acknowledgements}
	The authors declare that they have no conflict of interest. The second author was also partially supported by CNPq with grant 309026/2020-2. 
	
	%\section*{Author Contribution}
	%
	%For the authors the contribution are the same.
	%\section*{Conflict of Interest}
	%The authors declare that they have no conflict of interest.

\end{document}